\documentclass[12pt]{article}

\usepackage{amssymb,amsmath,amsthm,enumerate,framed}

\newcommand{\ovl}{\overline}
\newcommand{\n}{\noindent}
\newcommand{\vp}{\varepsilon}
\newcommand{\cl}[1]{{\mathcal{#1}}}
\newcommand{\bb}[1]{{\mathbb{#1}}}

\newcommand{\oneGN}{\cl{GN}^{(1)}}
\newcommand{\twoGN}{\cl{GN}}

\numberwithin{equation}{section}

\theoremstyle{plain}
\newtheorem{lem}{Lemma}[section]
\newtheorem{pro}[lem]{Proposition}
\newtheorem{thm}[lem]{Theorem}
\newtheorem{cor}[lem]{Corollary}

\theoremstyle{definition}
\newtheorem{defn}[lem]{Definition}
\newtheorem{exm}[lem]{Example}

\newtheorem{rem}[lem]{Remark}

\theoremstyle{remark}

\newcommand{\Tr}{\mathrm{Tr}}
\newcommand{\vnotimes}{\ \overline{\otimes}\ }
\newcommand{\IIi}{$\mathrm{II}_1$ }

\begin{document}

\title{\textsc{Groupoid Normalizers of Tensor Products}}
\date{}
\author{Junsheng Fang\\\normalsize\texttt{jfang@cisunix.unh.edu}
\and Roger R. Smith\thanks{Partially supported by a grant from the National Science Foundation}\\ 
\normalsize\texttt{rsmith@math.tamu.edu}\and 
Stuart A. White\\\normalsize\texttt{s.white@maths.gla.ac.uk}\and Alan D. Wiggins\\
\normalsize\texttt{ alan.d.wiggins@vanderbilt.edu}}

\maketitle

\begin{abstract}
We consider an inclusion $B\subseteq M$ of finite von Neumann algebras satisfying
$B'\cap M\subseteq B$. A partial isometry $v\in M$ is called a groupoid normalizer if 
$vBv^*,\ v^*Bv\subseteq B$. Given two such inclusions $B_i\subseteq M_i$, $i=1,2$, we find approximations to the groupoid normalizers of $B_1 \vnotimes B_2$ in $M_1\vnotimes M_2$, from which we deduce that the von Neumann algebra generated by the groupoid normalizers of the tensor product is equal to the tensor product of the von Neumann algebras generated by the groupoid normalizers. Examples are given to show that this can fail without the hypothesis $B_i'\cap M_i\subseteq B_i$, $i=1,2$. We also prove a parallel result where the groupoid normalizers are replaced by the intertwiners, those partial isometries $v\in M$ satisfying $vBv^*\subseteq B$ and $v^*v,\ vv^*\in B$. 
\end{abstract}

\section{Introduction}\label{gpsec1}

The focus of this paper is an inclusion $B\subseteq M$ of finite von~Neumann algebras. Such inclusions have a rich diverse history, first being studied by Dixmier \cite{Dixmier.Masa} in the context of maximal abelian subalgebras (masas) of \IIi factors. These inclusions provided the basic building blocks for the theory of subfactors developed  by Jones in \cite{Jones.Index} and today they are a key component  in the study of structral properties of \IIi factors using the deformation-rigidity techniques introduced  by Popa in \cite{Popa.Betti}. 

In \cite{Dixmier.Masa}, Dixmier introduced a classification of masas in \IIi factors using \emph{normalizers}, defining $\mathcal N_M(B)=\{u\text{ a unitary in }M:uBu^*=B\}$. A masa $B\subset M$ is \emph{Cartan} or \emph{regular} if these normalizers generate $M$ and \emph{singular} if $\mathcal N_M(B)\subset B$. Feldman and Moore demonstrated the importance of Cartan masas, and hence normalizers, in the study of \IIi factors, showing that inclusions of Cartan masas arise from measurable equivalence relations and that, up to orbit equivalence, these relations determine the resulting inclusion, \cite{Feldman.Moore1,Feldman.Moore2}.  

Given two inclusions $B_i\subset M_i$ of masas in \IIi factors, it is immediate  that an elementary tensor $u_1\otimes u_2$ of unitaries $u_i\in M_i$ normalizes the tensor product inclusion $B=B_1\vnotimes B_2\subset M=M_1\vnotimes M_2$ if and only if each $u_i$ normalizes $B_i$. As a simple consequence, the tensor product of Cartan masas is again Cartan. More generally, the operation of passing to the von Neumann algebra generated by the normalizers was shown to commute with the tensor product operation for masas inside \IIi factors, in the sense that the equality 
\begin{equation}\label{TP}
\mathcal N_{M_1}(B_1)''\vnotimes \mathcal N_{M_2}(B_2)''=\mathcal N_{M_1\vnotimes M_2}(B_1\vnotimes B_2)'',
\end{equation}
holds. This was proved when both masas are singular in \cite{Saw.StrongSing} and the general case was established by Chifan in \cite{Chifan.Normalisers}. Since the containment from left to right in (\ref{TP}) is immediate, the problem  in both cases is to eliminate the possiblity that some unexpected unitary in the tensor product normalizes $B_1\vnotimes B_2$. This  difficulty  was overcome in  \cite{Saw.StrongSing} and \cite{Chifan.Normalisers} by employing techniques of Popa \cite{Popa.StrongRigidity1} to analyse the basic construction algebra $\langle M,e_B\rangle$ of Jones \cite{Jones.Index}. Beyond the masa setting, (\ref{TP}) holds  when each $B_i$ satisfies $B_i'\cap M_i=\mathbb C1$, the defining property of irreducible subfactors.   When each $B_i$ has finite Jones index in $M_i$, the identity (\ref{TP}) can be deduced from results of \cite{Popa.Entropy}. The infinite index case  was established in \cite{Saw.NormalisersSubfactors}, where every normalizing unitary of such a tensor product of irreducible subfactors was shown to be of the form $w(v_1\otimes v_2)$, where $w$ is a unitary in $B_1\vnotimes B_2$ and each $v_i\in\mathcal N_{M_i}(B_i)$. Some other situations where (\ref{TP}) holds are discussed in \cite{Fang.CompletelySingular}.

For general inclusions $B_i\subseteq M_i$ of finite von Neumann algebras, the commutation identity (\ref{TP}) can fail. Indeed, taking each $M_i$ to be a copy of the $3\times 3$ matrices and each $B_i\cong \mathbb C\oplus\mathbb M_2(\mathbb C)$, one obtains inclusions with $\mathcal N_{M_i}(B_i)\subset B_i$, yet there are non-trivial normalizers of $B_1\otimes B_2$ inside $M_1\otimes M_2$. This is due to the presence of partial isometries $v$ in $M_i\setminus B_i$ with $vB_iv^*\subseteq B_i$ and $v^*B_iv\subseteq B_i$, as the non-trivial unitary normalizers of $B_1\otimes B_2$ can all be written in the form $\sum_jx_j(v_{1,j}\otimes v_{2,j})$, where $x_j$ lie in $B_1\otimes B_2$ and the $v_{i,j}$ are partial isometries with $v_{i,j}B_iv_{i,j}^*\subseteq B$ and $v_{i,j}^*B_iv_{i,j}\subseteq B_i$. Defining the \emph{groupoid normalizers} of a unital inclusion $B\subset M$ to be the set $\twoGN_{M}(B)=\{v\text{ a partial isometry in }M:vBv^*\subseteq B,v^*Bv\subseteq B\}$, the example discussed above satisfies the commutation identity
\begin{equation}\label{GNTP}
\twoGN_{M_1}(B_1)''\vnotimes\twoGN_{M_2}(B_2)''=\twoGN_{M_1\vnotimes M_2}(B_1\vnotimes B_2)''.
\end{equation}
In this paper we examine groupoid normalizers of tensor product algebras, establishing (Corollary \ref{TwoCor}) the identity (\ref{GNTP}) whenever $B_i\subseteq M_i$ are inclusions of finite von Neumann algebras with separable preduals satisfying $B_i'\cap M_i\subseteq B_i$ for each $i$. In \cite{Dye.Groups1} Dye shows that every groupoid normalizer $v$ of a masa $B$ in $M$ is of the form $v=ue$ for some projection $e=v^*v\in B$ and some unitary normalizer $u$ of $B$ in $M$, see also \cite[Lemma 6.2.3]{Sinclair.MasaBook}. The same result holds by a direct computation when $B$ is an irreducible subfactor of $M$, so that in these two cases $\mathcal N_M(B)''=\twoGN_M(B)''$ and (\ref{GNTP}) directly generalizes (\ref{TP}) established in \cite{Chifan.Normalisers} and \cite{Saw.NormalisersSubfactors} respectively.

The following example shows why the hypothesis $B_i'\cap M_i\subseteq B_i$ (which is satisfied by both masas and irreducible subfactors) is necessary in this result.
\begin{exm}\label{gpexm4.10}
 Consider the subalgebra
\[
 B = \left\{\begin{pmatrix}
             \alpha&0&0\\ 0&\alpha&0\\ 0&0&\beta
            \end{pmatrix}\colon \ \alpha,\beta\in {\bb C}\right\} \subseteq {\bb M}_3,
\]
and note that $B'\cap {\bb M}_3$ strictly contains $B$. A direct computation shows that $\twoGN (B)'' = {\bb M}_2\oplus {\bb C}$, and so $\twoGN_{\bb M_3}(B)''\otimes \twoGN_{\bb M_3}(B)''\cong{\bb M}_4 \oplus {\bb M}_2 \oplus {\bb M}_2 \oplus {\bb C}$. However, $B\otimes B$ is isomorphic to ${\bb C}I_4 \oplus {\bb C}I_2 \oplus {\bb C}I_2 \oplus {\bb C}$ inside ${\bb M}_9$, and $\twoGN_{\bb M_3\otimes \bb M_3} (B\otimes B)''$ is ${\bb M}_4 \oplus {\bb M}_4 \oplus {\bb C}$. 
\end{exm}

A new feature of \cite{Saw.NormalisersSubfactors} was the notion of \emph{one-sided normalizers} of an irreducible inclusion $B\subset M$ of \IIi factors, namely those unitaries $u\in M$ with $uBu^*\subsetneq B$. These cannot arise for finite index inclusions by index considerations, or in the case when $B\subset M$ is a masa. To establish (\ref{TP}) for irreducible subfactors, it was necessary to first establish the general form of a one-sided normalizer of a tensor product of irreducible subfactors and then deduce the normalizer result from this. The same procedure is necessary here, so we introduce the notion of an \emph{intertwiner} to study groupoid normalizers in a one-sided situation.
\begin{defn}
Given an inclusion $B\subseteq M$ of von Neumann algebras satisfying $B'\cap M\subseteq B$, define the collection $\oneGN_M(B)$ of \emph{intertwiners} of $B$ in $M$ by
$$
\oneGN_M(B)=\{v\text{ a partial isometry in }M:vBv^*\subseteq B,v^*v\in B\}.
$$
\end{defn}
\n We will write $\oneGN(B)$ for $\oneGN_M(B)$ when there is no confusion about the underlying algebra $M$. We use the superscript $^{(1)}$ to indicate that our intertwiners are \emph{one-sided}, namely that although $vBv^*\subseteq B$, we are not guaranteed to have a containment $v^*Bv\subseteq B$. Note that $v\in\twoGN_M(B)$ if, and only if, both $v$ and $v^*$ lies in $\oneGN_M(B)$. Note too that while the groupoid normalizers form a groupoid, the intertwiners do not. Finally, the terminology intertwiner comes from the fact that, under the hypothesis $B'\cap M\subseteq B$, these are exactly the partial isometries that witness the embeddability of a corner of $B$ into itself inside $M$ in the sense of Popa's intertwining procedure for subalgebras from \cite{Popa.Betti,Popa.StrongRigidity1}.

We obtain a similar commutation result to (\ref{GNTP}) for intertwiners. In fact our main theorem, stated below, obtains more as it gives approximate forms for intertwiners and groupoid normalizers of tensor products.
\begin{thm}\label{Main}
Let $B_i\subset M_i$ be inclusions of finite von Neumann algebras with separable preduals and with fixed faithful normal traces $\tau_i$ on $M_i$. For $v\in\oneGN_{M_1\vnotimes M_2}(B_1\vnotimes B_2)$ and $\vp>0$, there exist $k\in\mathbb N$ and operators $x_1,\dots,x_k\in B_1\vnotimes B_2$, intertwiners $w_{1,1},\dots,w_{1,k}$ of $B_1$ in $M_1$ and intertwiners $w_{2,1},\dots,w_{2,k}$ of $B_2$ in $M_2$ such that
\begin{equation}
\left\|v-\sum_{j=1}^kx_j(w_{1,j}\otimes w_{2,j})\right\|_2<\vp,
\end{equation}
where the $\|\cdot\|_2$-norm arises from the trace $\tau_1\otimes\tau_2$ on $M_1\vnotimes M_2$.  If in addition $v$ is a groupoid normalizer, then each $w_{i,j}$ can be taken to be a groupoid normalizer rather than just an intertwiner.
\end{thm}

The intertwiner form of Theorem \ref{Main} is established as Theorem \ref{gpthm4.8} and additional analysis in Section \ref{gpsec5} enables us to deduce the groupoid normalizer form of Theorem \ref{Main} as Theorem \ref{gpthm5.5}. For the remainder of the introduction we give a summary of the main steps used to establish these results and where they can be found in the paper.

Given inclusions $B_i\subseteq M_i$ of finite von Neumann algebras with $B_i'\cap M_i\subseteq B_i$, write $B\subset M$ for the tensor product inclusion $B_1\vnotimes B_2\subseteq M_1\vnotimes M_2$. Let $v\in\oneGN_M(B)$. Then the element $v^*e_Bv$ is a projection in the basic construction algebra $\langle M,e_B\rangle$, the properties of which are recalled in  Section \ref{gpsec2}. Section \ref{gpsec3} discusses the properties of these projections in the basic construction arising from intertwiners. In particular, we show that the projection $v^*e_Bv$ is central in the cutdown $(B'\cap\langle M,e_B\rangle)v^*v$ (Lemma \ref{gplem3.2}) and construct an explicit projection $P_v\in Z(B'\cap\langle M,e_B\rangle)$ with $P_vv^*v=v^*e_Bv$. We need to construct this projection explicitly rather than appeal to general theory, as its properties (established in Lemma \ref{gplem4.5}) are crucial subsequently. 

Since the basic construction factorizes as a tensor product $\langle M,e_B\rangle\cong\langle M_1,e_{B_1}\rangle\vnotimes\langle M_2,e_{B_2}\rangle$, Tomita's commutation theorem gives
\begin{equation}\label{TBC}
Z(B'\cap \langle M,e_B\rangle)\cong Z(B_1'\cap\langle M_1,e_{B_1}\rangle)\vnotimes Z(B_2'\cap\langle M_2,e_{B_2}\rangle).
\end{equation}
For each $i=1,2$, let $Q_i$ denote the supremum of all projections in $Z(B_i'\cap\langle M_i,e_{B_i}\rangle)$ of the form $\sum_jw_{i,j}^*e_{B_i}w_{i,j}$, where the $w_{i,j}$ lie in $\oneGN_{M_i}(B_i)$ and satisfy $w_{i,j}w_{i,k}^*=0$ when $j\neq k$. If we can show that 
\begin{equation}\label{IntroPLeqQ}
P_v\leq Q_1\otimes Q_2,
\end{equation}
then it will follow that we can approximate $P_v$ in $L^2(\langle M,e_B\rangle)$ by projections of the form $\sum_j(w_{1,j}\otimes w_{2,j})^*e_B(w_{1,j}\otimes w_{2,j})$ for intertwiners $w_{i,j}\in\oneGN_{M_i}(B_i)$. To do this, we use the fact that projections in the tensor product (\ref{TBC}) of \emph{abelian} von Neumann algebras can be approximated by sums of elementary tensors of projections, and so it is crucial that the original projection $v^*e_Bv$ be \emph{central} in $(B'\cap\langle M,e_B\rangle)v^*v$, for which the hypothesis $B'\cap M\subseteq B$ is necessary. Finally, we push the approximation for $P_v$ down to $M$ and obtain the required approximation for $v$ in $M$ (see Theorem \ref{gpthm4.8}).

Most of Section \ref{gpsec4} is taken up with establishing (\ref{IntroPLeqQ}). We give a technical result (Theorem \ref{gpthm4.1}), which in particular characterizes when a projection in the basic construction arises from an intertwiner. By applying Theorem \ref{gpthm4.1} to $P_v$ and the inclusion
$$
Z(B_1'\cap\langle M_1,e_{B_1}\rangle)\vnotimes B_2\subseteq Z(B_1'\cap\langle M_1,e_{B_1}\rangle)\vnotimes M_2,
$$
regarded as a direct integral of inclusions of finite von Neumann algebras, we are able to establish $P_v\leq 1\otimes Q_2$ in Lemma \ref{gplem4.7} and so (\ref{IntroPLeqQ}) follows by symmetry. It should be noted that  the introduction of the projections $Q_i$ is essential in order to make use of measure theory, particularly the uniqueness of product measures on $\sigma$--finite spaces, \cite[p. 312]{Roy}.  The canonical trace on the basic construction need not be a semifinite weight on $Z(B_i'\cap\langle M_i,e_{B_i}\rangle)$ but does have this property on the compression 
$Z(B_i'\cap\langle M_i,e_{B_i}\rangle)Q_i$ where it can be treated as a measure (see Lemma \ref{gplem4.4} and the discussion preceding Definition \ref{DefQ}). The remaining difficulty is to check that the projection $P_v$ satisfies the hypotheses of Theorem \ref{gpthm4.1}, for which we require certain  order properties of the pull-down map on the basic construction. These are described in the next section, in which we also set out our notation, review 
the properties of the basic construction, and establish some technical lemmas. Finally, the paper ends with Section \ref{gpsec5}, which handles the additional details required to deduce the groupoid normalizer result (Theorem \ref{gpthm5.5}) from our earlier work.

\textbf{Acknowledgment: } The work in this paper originated during the \emph{Workshop in Analysis and Probability}, held at Texas A\& M University during Summer 2007.  It is a pleasure to express our thanks to both the organizers of the workshop and to the NSF for providing financial support to the workshop.

\section{Notation and preliminaries}\label{gpsec2}

\indent 

Throughout the paper, all von Neumann algebras are assumed to have separable preduals. The basic object of study in this paper is an inclusion $B\subseteq M$ of finite von Neumann algebras, where $M$ is equipped with a faithful normal trace $\tau$ satisfying $\tau(1)=1$. We always assume that $M$ is standardly represented on the Hilbert space $L^2(M,\tau)$, or simply $L^2(M)$. The letter $\xi$ is reserved for the image of $1\in M$ in this Hilbert space, and $J$ will denote the isometric conjugate linear operator on $L^2(M)$ defined on $M\xi$ by $J(x\xi) = x^*\xi$, $x\in M$, and extended by continuity to $L^2(M)$ from this dense subspace. Then $L^2(B)$ is a closed subspace of $L^2(M)$, and $e_B$ denotes the projection of $L^2(M)$ onto $L^2(B)$, called the Jones projection. The von~Neumann algebra generated by $M$ and $e_B$ is called the basic construction and is denoted by $\langle M,e_B\rangle$, \cite{Chr,Jones.Index}.  Let ${\bb E}_B$ denote the unique trace preserving conditional expectation of $M$ onto $B$. In the next proposition we collect together standard properties of $e_B,\mathbb E_B$ and $\langle M,e_B\rangle$ from \cite{Jones.Index,Popa.Entropy,Jones.SubfactorsBook,Sinclair.MasaBook}.

\begin{pro}\label{ListOfFacts}
\begin{enumerate}[(i)]
 \item $e_B(x\xi) = {\bb E}_B(x)\xi,\ x\in M$.
\item $e_Bxe_B = {\bb E}_B(x)e_B = e_B{\bb E}_B(x), \ x\in M$.
\item $M\cap \{e_B\}'=B$.
\item $\langle M,e_B\rangle'=JBJ,\quad Z(\langle M,e_B\rangle)=JZ(B)J$.
\item $e_B$ has central support 1 in $\langle M,e_B\rangle$.
\item $\mathrm{Span}\{xe_By\colon \ x,y\in M\}$ generates a $*$-strongly dense subalgebra, denoted $Me_BM$, of $\langle M,e_B\rangle$.
\item $x\mapsto e_Bx$ and $x\mapsto xe_B$ are injective maps for $x\in M$.
\item $Me_B$ and $e_BM$ are $*$-strongly dense in $\langle M,e_B\rangle e_B$ and $e_B\langle M,e_B\rangle$ respectively.
\item $e_B\langle M,e_B\rangle e_B = Be_B = e_BB$.
\item $(Me_BM) \langle M,e_B\rangle (Me_BM)\subseteq Me_BM$.
\item There is a unique faithful normal semifinite trace $\Tr$ on $\langle M,e_B\rangle$ satisfying
\begin{equation}\label{gpeq2.1}
\Tr(xe_By) = \tau(xy),\qquad x,y\in M.
\end{equation}
This trace is given by the formula
\begin{equation}\label{DefTrace}
\Tr(t)=\sum_{i=1}^\infty\left<tJv_i^*\xi,Jv_i^*\xi\right>,t\in\langle M,e_B\rangle^+,
\end{equation}
where the $v_i$'s are partial isometries in $\langle M,e_B\rangle$ satisfying $\sum_{i=1}^\infty v_i^*e_Bv_i=1$.
\item The algebra $Me_BM$ is $\|\cdot\|_{2,\Tr}$-dense in $L^2(\langle M,e_B\rangle,\Tr)$ and $\|\cdot\|_{1,\Tr}$-dense in $L^1(\langle M,e_B\rangle,\Tr)$.
\item Given inclusions $B_i\subset M_i$ of finite von Neumann algebras for $i=1,2$, the basic construction $\langle M_1\vnotimes M_2,e_{B_1\vnotimes B_2}\rangle$ is isomorphic to $\langle M_1,e_{B_1}\rangle\vnotimes\langle M_2,e_{B_2}\rangle$. Under this isomorphism, the canonical trace $\Tr$ on $\langle M_1\vnotimes M_2,e_{B_1\vnotimes B_2}\rangle$ is given by $\Tr_1\otimes\Tr_2$, where $\Tr_i$ is the canonical trace on $\langle M_i,e_{B_i}\rangle$.
\item There is a well defined map $\Psi:Me_BM\rightarrow M$, given by
\begin{equation}\label{gpeq2.2}
 \Psi(xe_By) = xy,\qquad x,y\in M.
\end{equation}
This is the pull down map of \cite{Popa.Entropy}, where it was shown to extend to a contraction from $L^1(\langle M,e_B\rangle,\Tr)$ to $L^1(M,\tau)$. 
\end{enumerate}
\end{pro}

Using Part (xii) of the previous proposition, the equation
\begin{equation}\label{gpeq2.3}
\Tr((xe_By)z) = \tau((xy)z) = \tau(\Psi(xe_By)z),\qquad x,y,z\in M,
\end{equation}
shows that $\Psi$ is the pre-adjoint of the identity embedding $M\hookrightarrow \langle M,e_B\rangle$ and is, in particular, positive. The basic properties of $\Psi$ are set out in \cite{Popa.Entropy}, but we will need more detailed information on this map than is currently available in the literature. We devote much of this section to obtaining further properties of $\Psi$, the main objective being to apply them in Lemma \ref{gplem4.6}.

In the next three lemmas, the inclusion $B\subset M$ is always of arbitrary finite von Neumann algebras with a fixed faithful normalized normal trace $\tau$ on $M$, inducing the trace $\Tr$ on $\langle M,e_B\rangle$. 
\begin{lem}\label{gplem2.2}
 Let $x\in L^1(\langle M,e_B\rangle)^+ \cap \langle M,e_B\rangle$. If $\Psi(x)\in L^1(M)\cap M$, then $\Psi(x)\ge x$.
\end{lem}

\begin{proof}
It suffices to show that
\begin{equation}\label{gpeq2.4}
 \langle \Psi(x) y\xi,y\xi\rangle \ge \langle xy\xi, y\xi\rangle, \qquad y\in M.
\end{equation}
The maximality argument, preceding \cite[Lemma 4.3.4]{Sinclair.MasaBook}, to establish part (xi) of Proposition \ref{ListOfFacts} can be easily modified to incorporate the requirement that $v_1=1$. Thus there are vectors $\xi_i = Jv^*_i\xi\in L^2(M)$ so that \eqref{DefTrace} becomes
\begin{equation}\label{gpeq2.7}
\Tr(t) = \sum^\infty_{i=1} \langle t\xi_i,\xi_i\rangle,\qquad t\in \langle M,e_B\rangle^+,
\end{equation}
where $\xi_1=\xi$. Now, for $y\in M$, we may use the $M$-modularity of $\Psi$ to write
\begin{equation}\label{gpeq2.8}
\langle\Psi(x)y\xi,y\xi\rangle = \langle\Psi(y^*xy)\xi,\xi\rangle= \tau(\Psi(y^*xy))=\Tr(y^*xy).
\end{equation}
It follows from \eqref{gpeq2.7} and \eqref{gpeq2.8} that
\begin{equation}\label{gpeq2.9}
 \langle\Psi(x)y\xi,y\xi\rangle = \langle xy\xi,y\xi\rangle + \sum^\infty_{i=2} \langle xy\xi_i, y\xi_i\rangle\ge \langle xy\xi, y\xi\rangle,\qquad y\in M,
\end{equation}
establishing that $\Psi(x)\ge x$.
\end{proof}

We now extend this result to tensor products.  Let $N$ be a semifinite von Neumann algebra with a specified faithful normal semifinite trace $\mathrm{TR}$. In \cite{Effros.ApproximationProperties}, Effros and Ruan identified the predual of a tensor product of von~Neumann algebras $X$ and $Y$ by $(X\vnotimes Y)_* = X_* \otimes_{op} Y_*$, the operator space projective tensor product of the preduals. In the presence of traces, this identifies $L^1(X\vnotimes Y)$ with $L^1(X)\otimes_{op} L^1(Y)$, so $I\otimes \Psi$ is well defined, positive, and bounded from $L^1(N\vnotimes \langle M,e_B\rangle,\mathrm{TR}\otimes\Tr)$ to $L^1(N\vnotimes M,\mathrm{TR}\otimes\tau)$, being the pre-adjoint of the identity embedding $N\vnotimes M\hookrightarrow N\vnotimes \langle M,e_B\rangle$. Following \cite[Chapter IX]{Takesaki.2}, we will always assume that $N$ is faithfully represented on $L^2(N,\mathrm{TR})$, for which $\text{span}\{y\in N\colon \ \mathrm{TR}(y^*y)<\infty\}$ is a dense subspace.

\begin{lem}\label{gplem2.3}
Let $x\in L^1(N\vnotimes \langle M,e_B\rangle)^+\cap (N\vnotimes \langle M,e_B\rangle)$. If $(I\otimes \Psi)(x) \in L^1(N\vnotimes M)\cap (N\vnotimes M)$, then $(I\otimes \Psi)(x)\ge x$.
\end{lem}

\begin{proof}
Suppose that the result is not true. Then we may find a finite projection $p\in N$, elements $y_i\in pNp$ and $z_i\in M$, $1\le i\le k$, so that
\begin{equation}\label{gpeq2.10}
 \left\langle(x-(I\otimes\Psi)(x)) \sum^k_{i=1} y_i\otimes z_i\xi, \sum^k_{i=1} y_i\otimes z_i\xi \right\rangle > 0,
\end{equation}
since such sums $\sum_{i=1}^ky_i\otimes z_i\xi$ are dense in $L^2(N,\mathrm{TR})\otimes_2 L^2(M,\tau)$. 
Then the inequality
\begin{equation}\label{gpeq2.11}
 (I\otimes\Psi) \left((p\otimes 1)x(p\otimes 1)\right) \ge (p\otimes 1)x(p\otimes 1)
\end{equation}
fails. The element on the left of \eqref{gpeq2.11} is $(p\otimes 1)\left((I\otimes \Psi)(x)\right) (p\otimes 1)$, and so is bounded by hypothesis. The restriction of $I\otimes\Psi$ to $L^1(pNp\vnotimes \langle M,e_B\rangle)$ is the pull down map for the inclusion $pNp \vnotimes B\subseteq pNp\vnotimes M$ of finite von~Neumann algebras with basic construction $pNp\vnotimes \langle M,e_B\rangle$. The failure of \eqref{gpeq2.11} then  contradicts Lemma \ref{gplem2.2} applied to this inclusion, establishing that $(1\otimes\Psi)(x)\ge x$.
\end{proof}

The next lemma completes our investigation of the order properties of pull down maps.

\begin{lem}\label{gplem2.4}
 If $x\in L^1(N\vnotimes \langle M,e_B\rangle)^+$ is unbounded, then so also is $(1\otimes\Psi)(x)$.
\end{lem}

\begin{proof}
Suppose that $(1\otimes\Psi)(x)$ is bounded. Following \cite[Section IX.2]{Takesaki.2}, we may regard $x$ as a self-adjoint positive densely defined operator on $L^2(N\vnotimes\langle M,e_B\rangle)$. For $n\ge 1$, let $p_n\in N$ be the spectral projection of $x$ for the interval $[0,n]$. Then $p_nx\le x$, so $(I\otimes\Psi(p_nx))\le (I\otimes\Psi)(x)$, since $I\otimes \Psi$ is the pre-adjoint of a positive map. In particular, $I\otimes\Psi(p_nx)$ is bounded. By Lemma \ref{gplem2.3} applied to $p_nx$,
\begin{equation}\label{gpeq2.12}
 (I\otimes\Psi)(x) \ge(I\otimes\Psi)(p_nx) \ge p_nx.
\end{equation}
Since $n\ge 1$ was arbitrary, we conclude from \eqref{gpeq2.12} that $x$ is bounded, a contradiction which completes the proof.
\end{proof}

We note for future reference that these results are equally valid for pull down maps of the form $\Psi\otimes I$, due to symmetry. These lemmas will be used in  Section \ref{gpsec4} to derive an important inequality. The next lemma formulates exactly what will be needed.

\begin{lem}\label{gplem2.5}
 Let $B_i\subseteq M_i$, $i=1,2$, be inclusions of finite von Neumann algebras with pull down maps $\Psi_i$. Let $B\subseteq M$ be the inclusion $B_1\vnotimes  B_2\subseteq M_1\vnotimes M_2$. If $x\in L^1(\langle M,e_B\rangle)^+\cap \langle M,e_B\rangle$ is such that $(\Psi_1\otimes\Psi_2) (x) \in L^1(M)\cap M$ and the inequality
\begin{equation}\label{gpeq2.13}
 \|(\Psi_1\otimes\Psi_2)(x)\|\le 1
\end{equation}
is satisfied, then $(I\otimes\Psi_2)(x)\in L^1(\langle M_1,e_{B_1}\rangle \vnotimes 
M_2)^+\cap (\langle M_1,e_{B_1}\rangle\vnotimes M_2)$ and $\|(I\otimes \Psi_2)(x)\|\le 1$.
\end{lem}

\begin{proof}
Using the isomorphism of Proposition \ref{ListOfFacts} (xiii), $\Psi_1\otimes \Psi_2$ is the pull down map for $\langle M,e_B\rangle$. Since $(\Psi_1\otimes I) ((I\otimes \Psi_2)(x)) = (\Psi_1\otimes \Psi_2)(x)$ is a bounded operator by hypothesis, it follows from Lemma \ref{gplem2.4} that $(I\otimes \Psi_2)(x)$ is also bounded in $\langle M_1,e_{B_1}\rangle \vnotimes M_2$. Thus the three operators $x$, $(I\otimes \Psi_2)(x)$ and $(\Psi_1\otimes \Psi_2)(x)$ are all bounded, and so we may apply Lemma \ref{gplem2.3} twice to the pull down maps $I\otimes \Psi_2$ and $\Psi_1\otimes I$ to obtain
\begin{equation}\label{gpeq2.14}
 (\Psi_1\otimes \Psi_2)(x) \ge (I\otimes \Psi_2)(x) \ge x.
\end{equation}
The result then follows from \eqref{gpeq2.14} and the hypothesis \eqref{gpeq2.13}.
\end{proof}

In the proof of Lemma \ref{gplem4.7}, we will need the following fact regarding inclusions of finite von Neumann algebras $B\subset M$ with $B'\cap M\subseteq B$. Here, and elsewhere in the paper, we consider inclusions induced by cut-downs. Recall that if $Q\subseteq N$ is an inclusion of von Neumann algebras and $q$ is a projection in $Q$, then
\begin{equation}\label{Cutdown}
(Q'\cap N)q=(qQq)'\cap (qNq),\quad Z((Q'\cap N)q)=Z((qQq)'\cap (qNq)),
\end{equation}
see, for example, \cite[Section 5.4]{Sinclair.MasaBook}.

 \begin{lem}\label{gplem4.4}
Let $B\subseteq M$ be a containment of finite von Neumann algebras such that $B'\cap M \subseteq B$. If $p\in M$ is a nonzero projection, then there exists a nonzero projection $q\in B$ which is equivalent to a subprojection of $p$. 
\end{lem}
Observe that if $M$ is a finite factor, then Lemma \ref{gplem4.4} is immediate.  Our proof of Lemma \ref{gplem4.4} is classical, proceeding by analysing the center-valued trace on $P$. Alternatively one can establish the lemma by taking a direct integral over the center.   Since we have been unable to find this fact in the literature we give the details for completeness.

\begin{proof}[Proof of Lemma \ref{gplem4.4}.]
Let $\Delta$ denote the center-valued trace on $M$. We will make use of two properties of $\Delta$ from \cite[Theorem 8.4.3]{KR.2}. The first is that $p_1\precsim p_2$ if and only if $\Delta(p_1) \le \Delta(p_2)$, and the second is that $p_1\sim p_2$ if and only if $\Delta(p_1) = \Delta(p_2)$.

The hypothesis $B'\cap M\subseteq B$ implies that $B'\cap M = Z(B)$ and, in particular, that $Z(M) \subseteq Z(B)$. For some sufficiently small $c>0$, the spectral projection $z$ of $\Delta(p)$ for the interval $[c,1]$ is nonzero, and $\Delta(pz)\ge cz$. Since $Bz\subseteq Mz$ also satisfies the relative commutant hypothesis, it suffices to prove the result under the additional restriction $\Delta(p) \ge c1$ for some constant $c>0$.

Let $n\ge c^{-1}$ be any integer. Suppose that it is possible to find a nonzero projection $q\in B$ and an orthogonal set $\{q,p_2,\ldots, p_n\}$ of equivalent projections in $M$. The sum of these projections has central trace equal to $n\Delta(q)$ and is also bounded by 1, so that $\Delta(q)\le n^{-1}1\le c1$. But then $q\precsim p$ and we are done. Thus we may assume that there is an absolute bound on the length of any such set, and we may then choose one, $\{q_1,p_2,\ldots, p_n\}$, of maximal length. By cutting by the central support of $q_1$, we may assume that this 
central support is 1.

Now consider the inclusion $q_1Bq_1\subseteq q_1Mq_1$, and note that
\begin{equation}\label{gpeq4.23}
 (q_1Bq_1)'\cap q_1Mq_1 = q_1(B'\cap M) = q_1Z(B) = Z(q_1Bq_1)
\end{equation}
from (\ref{Cutdown}). Let $f_1$ and $f_2$ be non-zero orthogonal projections in $q_1Bq_1$ and $q_1Mq_1$ respectively. By the comparison theory of projections, there exists a projection $z\in Z(q_1Mq_1) \subseteq Z(q_1Bq_1)$ so that
\begin{equation}\label{gpeq4.24}
zf_1 \precsim zf_2,\quad (1-z)f_2\precsim (1-z)f_1,
\end{equation}
the equivalence being taken in $q_1Mq_1$. Now $zf_1\in q_1Bq_1$ and is equivalent to a subprojection $p_0$ of $zf_2\le q_1$. Then the pair $zf_1,p_0$ is equivalent to orthogonal pairs below each $p_i$, $2\le i\le n$, which will contradict the maximal length of $\{q_1,p_2, \ldots, p_n\}$ unless $zf_1=0$. Similarly $(1-z)f_2 = 0$. Thus $f_1$ and $f_2$ have orthogonal central supports in $q_1Bq_1$ and so \cite[Lemma 5.5.3]{Sinclair.MasaBook} shows that $q_1Bq_1$ is abelian. Equation (\ref{gpeq4.23}) then shows that $q_1Bq_1$ is a masa in $q_1Mq_1$, and so another application of \cite[Lemma 5.5.3]{Sinclair.MasaBook} shows that $q_1Mq_1$ is also abelian. Thus $q_1Bq_1=q_1Mq_1$.

Now the projection $1-q_1-p_2 \cdots -p_n$ must be $0$, otherwise it would have a non-zero subprojection equivalent to a nonzero projection $\tilde q_1 \in q_1Mq_1 = q_1Bq_1$, since $q_1$ has central support 1, and $\tilde q_1$ would lie in a set of $n+1$ equivalent orthogonal projections. Thus $q_1,p_2,\ldots, p_n$ are abelian projections in $M$ with sum 1, so $M$ is isomorphic to $L^\infty(\Omega)\otimes {\bb M}_n$ for some measure space $\Omega$. Identify $p$ and $q_1$ with measurable ${\bb M}_n$-valued functions. Since $q_1$ is abelian, the rank of $q_1(\omega)$ is 1 almost everywhere, and the rank of $p(\omega)$ is at least 1 almost everywhere since $\Delta(p)\ge c1$. Then $q_1$ is equivalent to a subprojection of $p$ since $\Delta(q_1) \le \Delta(p)$. This completes the proof.
\end{proof}

We conclude this section with a brief explanation of an averaging technique in finite von Neumann algebras which we will use subsequently. It has its origins in \cite{Chr}, but is also used extensively in \cite{Popa.Betti}. If $\eta\in L^2(M)$ and $\mathcal{U}$ is a group of unitaries in $M$ then the vector can be averaged over ${\cl U}$. This is normally associated with amenable groups, but can be made to work in this setting without this assumption. Form the $\|\cdot\|_2$-norm closure $\ovl K$ of
\[
K = \text{conv}\{u\eta u^*\colon \ u\in {\cl U}\}.
\]
There is a unique vector $\tilde\eta\in \ovl K$ of minimal norm, and uniqueness of $\tilde\eta$ implies that $u\tilde \eta u^* = \tilde \eta$ for all $u\in {\cl U}$. We refer to $\tilde\eta$ as the result of averaging $\eta$ over ${\cl U}$, and many variations of this are possible.  We give an example of this technique by establishing a technical result which will be needed in the proof of Theorem \ref{gpthm4.8}.

Recall (Proposition \ref{ListOfFacts}, (xii)) that $Me_BM$ is $\|\cdot\|_{2,\Tr}$-dense in $L^2(\langle M,e_B\rangle,\Tr)$. Consider a $*$-subalgebra $A$  which is strongly dense in $M$. If $x,y\in M$, then fix  sequences $\{x_n\}^\infty_{n=1}$, $\{y_n\}^\infty_{n=1}$ from $A$ converging strongly to $x$ and $y$, respectively. Then
\begin{align}
 \|(x-x_n)e_B\|^2_{2,\Tr} &=\Tr(e_B(x-x_n)^* (x-x_n)e_B)\notag\\
& = \tau((x-x_n)^*(x-x_n))=\langle (x-x_n)\xi, (x-x_n)\xi\rangle,
\end{align}
so $x_ne_B\to xe_B$ in $\|\cdot\|_{2,\Tr}$-norm. Thus $x_ne_By \to xe_By$ so, given $\vp>0$, we may choose $n_0$ so large that $\|x_{n_0}e_By- xe_By\|_{2,\Tr} < \vp/2$. The same argument on the right allows us to choose $n_1$ so large that $\|x_{n_0}e_By- x_{n_0}e_By_{n_1}\|_{2,\Tr} < \vp/2$, whereupon $\|xe_By - x_{n_0}e_By_{n_1}\|_{2,\Tr} < \vp$. The conclusion reached is that the algebra $Ae_BA=\{\sum_{i=1}^nx_ie_Ny_i:x_i,y_i\in A\}$ is $\|\cdot\|_{2,\Tr}$-norm dense in $L^2(\langle M,e_B\rangle,\Tr)$. In the next lemma, we will use this when $M$ is a tensor product $M_1\vnotimes M_2$ where we take $A$ to be the algebraic tensor product $M_1\otimes M_2$.

\begin{lem}\label{gplem3.7}
 Let $B_1,B_2$ be von Neumann subalgebras of finite von Neumann algebras $M_1,M_2$ and let $B = B_1\vnotimes B_2$, $M = M_1\vnotimes M_2$. Then
\[
 L^2(Z(B'\cap \langle M,e_B\rangle),\Tr) = L^2(Z(B'_1\cap \langle M_1,e_{B_1}\rangle), \Tr_1) \otimes_2 L^2(Z(B'_2\cap \langle M_2,e_{B_2}\rangle),\Tr_2).
\]
\end{lem}

Note that, although $\Tr_i$ is a semifinite trace on $M_i$, it need not be semifinite on $Z(B_i'\cap\langle M_i,e_{B_i}\rangle)$. This is why the lemma cannot be obtained immediately from the uniqueness of product measures on $\sigma$-finite measure spaces.

\begin{proof}[Proof of Lemma \ref{gplem3.7}.]
If $z_i\in Z(B'_i\cap \langle M_i,e_{B_i}\rangle)$, $i=1,2$, then $z_1\otimes z_2\in Z(B'\cap \langle M,e_B\rangle)$ and $\|z_1\otimes z_2\|_{2,\Tr} = \|z_1\|_{2,\Tr_1} \|z_2\|_{2, \Tr_2}$. This shows the containment from right to left.

Suppose that $z\in Z(B'\cap \langle M,e_B\rangle)$ with $\Tr(z^*z)<\infty$. Then $z$ lies in $L^2(\langle M,e_B\rangle,\Tr)$ so can be approximated in $\|\cdot\|_{2,\Tr}$-norm by sums of the form $\sum\limits^k_{i=1} x_ie_By_i$ with $x_i,y_i\in M_1\vnotimes M_2$. The preceding remarks then allow us to assume that $x_i$ and $y_i$ lie in the algebraic tensor product $M_1\otimes M_2$. Thus, given $\vp>0$, we may find elements $a_i,c_i\in M_1$, $b_i,d_i\in M_2$ so that
\begin{equation}\label{gpeq3.32}
 \left\|z - \sum^n_{i=1} (a_i\otimes b_i)(e_{B_1}\otimes e_{B_2}) (c_i\otimes d_i)\right\|_{2,\Tr} \le \vp.
\end{equation}
This may be rewritten as
\begin{equation}\label{gpeq3.33}
 \left\|z - \sum^n_{i=1} (a_ie_{B_1}c_i) \otimes (b_ie_{B_2}d_i)\right\|_{2,\Tr} \le \vp,
\end{equation}
and then as
\begin{equation}\label{gpeq3.34}
 \left\|z - \sum^n_{i=1} f_i\otimes g_i\right\|_{2,\Tr} \le \vp,
\end{equation}
where $f_i\in M_1 e_{B_1}M_1$ and $g_i\in M_2e_{B_2}M_2$. We may further suppose that the set $\{g_1,\ldots, g_n\}$ is linearly independent.

For $j=1,2$, let $N_j$ be the von Neumann algebra generated by $B_j$ and $B'_j \cap \langle M_j,e_{B_j} \rangle$, and note that $z$ commutes with $N_1\ovl\otimes N_2$. Let
\[
K = \ovl{\rm conv}^w \left\{\sum^n_{i=1} uf_iu^* \otimes g_i\colon \ u\in {\cl U}(N_1)\right\}, \ \ K_i = \ovl{\rm conv}^w\{uf_iu^*\colon \ u\in {\cl U}(N_1)\}
\]
for $1\le i\le n$. Then $K \subseteq \sum\limits^n_{i=1} K_i\otimes g_i$. By \cite[Lemma 9.2.1]{Sinclair.MasaBook} $K$ and each $K_i$ are closed in their respective $\|\cdot\|_2$-norms. If $k\in K$ is the element of minimal $\|\cdot\|_2$-norm in $K$ then it may be written as $k = \sum\limits^n_{i=1}k_i\otimes g_i$ with $k_i\in K_i$. Since $k$ is invariant for the action of ${\cl U}(N_1\otimes 1)$, we see that
\begin{equation}\label{gpeq3.35}
 \sum^n_{i=1} (uk_iu^*-k_i) \otimes g_i = 0,\qquad u\in {\cl U}(N_1).
\end{equation}
The linear independence of the $g_i$'s allows us to conclude that $uk_iu^* = k_i$ for $1\le i\le n$ and $u\in {\cl U}(N_1)$. Thus $k_i \in N'_1\cap \langle M_1,e_{B_1}\rangle = Z(B'_1\cap \langle M_1,e_{B_1}\rangle)$. The inequality \eqref{gpeq3.34} is preserved by averaging in this manner over ${\cl U}(N_1\otimes 1)$ so, replacing each $f_i$ by $k_i$ if necessary, we may assume that $f_i\in Z(B'_1\cap \langle M_1,e_{B_1}\rangle)$ for  $1\le i\le n$.  Now repeat this argument on the right, averaging over ${\cl U}(1\otimes N_2)$, to replace the $g_i$'s by elements of $Z(B'_2\cap \langle M_2,e_{B_2}\rangle)$. With these changes, \eqref{gpeq3.34} now approximates $z$ by a sum from
\[
L^2(Z(B'_1\cap \langle M_1,e_{B_1}\rangle),\Tr_1) \otimes_2 L^2(Z(B'_2\cap \langle M_2, e_{B_2}\rangle), \Tr_2)
\]
which proves the containment from left to right and establishes equality.
\end{proof}

\section{Projections in the basic construction}\label{gpsec3}

\indent 

In this section, we relate intertwiners of a subalgebra to certain projections in the basic construction. We consider a finite von~Neumann algebra $M$ and a von~Neumann subalgebra $B$ whose unit will always coincide with that of $M$. For the most part, we will be interested in the condition $B'\cap M \subseteq B$ (equivalent to $B'\cap M = Z(B)$), but we will make this requirement explicit when it is needed.

\begin{lem}\label{gplem3.1}
Let $B$ be a von Neumann subalgebra of a finite von Neumann algebra $M$ and let $v\in \oneGN (B)$. 
\begin{enumerate}
\item Then $v^*e_Bv$ is a projection in $(B'\cap \langle M,e_B\rangle)v^*v$.  
\item Suppose $q$ is a projection in $B$. Then $v^*e_Bv$ lies in $(B' \cap\langle M,e_B\rangle)q$ if, and only if, $v^*v\in Z(B)q=Z(qBq)$. 
\end{enumerate}
\end{lem}

\begin{proof}
1.~~ The element $v^*e_Bv$ is positive in $\langle M,e_B\rangle$. Since $vv^*\in B$ and so commutes with $e_B$, the following calculation establishes that $v^*e_Bv$ is a projection:
\begin{equation}\label{gpeq3.1}
 (v^*e_Bv)^2 = v^*e_Bvv^*e_Bv = v^*vv^*e_Bv = v^*e_Bv.
\end{equation}
For an arbitrary $b\in v^*vBv^*v$,
\begin{equation}\label{gpeq3.2}
 (v^*e_Bv)b = v^*e_Bvbv^*v = v^*vbv^*e_Bv= v^*vbv^*(vv^*)e_Bv = b(v^*e_Bv),
\end{equation}
where the second equality uses $vbv^*\in B$ to commute this element with $e_B$. Thus \eqref{gpeq3.2} establishes that $v^*e_Bv\in (v^*vBv^*v)'\cap v^*v\langle M,e_B\rangle v^*v$ which is $(B'\cap \langle M,e_B\rangle)v^*v$ (see (\ref{Cutdown})).

\n2.~~ Suppose now that $v^*v\in Z(B)q=Z(qBq)$. It is immediate that $Z(qBq) \subseteq Z((B'\cap \langle M,e_B\rangle)q)$, so we have a decomposition
\begin{equation}\label{gpeq3.3}
(B'\cap \langle M,e_B\rangle)q = (B' \cap \langle M,e_B\rangle) v^*v \oplus (B'\cap \langle M,e_B\rangle) (q-v^*v).
\end{equation}
We have already shown that $v^*e_Bv$ is in the first summand of \eqref{gpeq3.3} so must lie in $(B'\cap \langle M,e_B\rangle)q$.

Conversely, the hypothesis on $v^*e_Bv$ implies that $v^*e_Bv = v^*e_Bvq = qv^*e_Bv$, so the pull down map gives $v^*v = v^*vq = qv^*v$, showing that $v^*v\in qBq$. For each $b\in B$, $v^*e_Bvqbq=qbqv^*e_Bv$. Applying the pull down map gives $v^*vqbq=qbqv^*v$ and hence $v^*v\in Z(qBq)$.
\end{proof}

We now strengthen this lemma under the additional hypothesis that $B'\cap M\subseteq B$. Recall from part (iv) of Proposition \ref{ListOfFacts} that $JZ(B)J = Z(\langle M,e_B\rangle)$. When $B\subset M$ is a finite index inclusion of irreducible subfactors, Lemma \ref{gplem3.2} is contained in \cite[Proposition 1.7 (2) and Proposition 1.9]{Popa.Entropy}. The proof follows the extension to infinite index inclusions of irreducible subfactors in \cite[Lemma 3.3]{Saw.NormalisersSubfactors}.

\begin{lem}\label{gplem3.2}
Let $B$ be a von Neumann subalgebra of a finite von~Neumann algebra $M$ and suppose that $B'\cap M\subseteq B$. Let $v\in \oneGN (B)$. Then the projection $v^*e_Bv$ is central in $(B'\cap \langle M,e_B\rangle)v^*v$.
\end{lem}

\begin{proof}
Define two projections $p,q\in B$ by $p=v^*v$ and $q=vv^*$. Now consider an arbitrary $x\in (B'\cap \langle M,e_B\rangle) p = (pBp)'\cap p\langle M,e_B\rangle p$. Then, for each $b\in B$,
\begin{equation}\label{gpeq3.6}
 (vxv^*)(vbv^*) = vxpbpv^* = vpbpxv^* = (vbv^*)(vxv^*),
\end{equation}
showing that $vxv^*\in (vBv^*)'\cap (q\langle M,e_B\rangle q)$. We next prove that $qe_B$ is central in $(vBv^*)'\cap (q\langle M,e_B\rangle q)$. It lies in this algebra by the previous calculation and Lemma \ref{gplem3.1}, as $qe_B=v(v^*e_Bv)v^*$.

Take $t\in (vBv^*)' \cap (q\langle M,e_B\rangle q)$ to be self-adjoint and let $\eta$ be $t\xi\in L^2(M)$. Now take a sequence $\{x_n\}^\infty_{n=1}$ from $M$ converging in $\|\cdot\|_2$-norm to $t\xi$. Since $t=qt=tq$, we may assume that the sequence $\{x_n\}^\infty_{n=1}$ lies in $qMq$, otherwise replace it by $\{qx_nq\}^\infty_{n=1}$.

For each $u$ in the unitary group ${\cl U}(pBp), t$ commutes with $vuv^*$ and so
\begin{align}
 Jvuv^*Jvuv^*\eta &= Jvuv^*Jvuv^*t\xi = Jvuv^*Jtvuv^*\xi\notag\\
&= tJvuv^* Jvuv^*\xi = tvuv^*vu^*v^*\xi\notag\\
\label{gpeq3.8}
&= tq\xi = t\xi = \eta,
\end{align}
where the third equality holds because $vuv^* \in B$ so that $Jvuv^*J\in (\langle M,e_B\rangle)'$. For each $n\ge 1$ and each $u\in {\cl U}(pBp)$,
\begin{align}
 \|Jvuv^*Jvuv^*x_n\xi - \eta\|_2 &= \|Jvuv^*Jvuv^*(x_n\xi-\eta)\|_2\notag\\
\label{gpeq3.9}
&\le \|x_n\xi-\eta\|_2,
\end{align}
from \eqref{gpeq3.8}. If we let $y_n$ be the element of $qMq$ obtained by averaging $x_n$ over the unitary group $v{\cl U}(pBp)v^* \subseteq qBq$, then \eqref{gpeq3.9} gives
\begin{equation}\label{gpeq3.10}
 \|y_n\xi-\eta\|_2 \le \|x_n\xi-\eta\|_2,\qquad n\ge 1,
\end{equation}
while $y_n\in (vBv^*)'\cap qMq$. Since
\begin{align}
 (vBv^*)' \cap qMq &= (vBv^*)'\cap (vMv^*) = v(B'\cap M)v^*\notag\\
\label{gpeq3.11}
&= vZ(B) v^*\subseteq qBq,
\end{align}
we see that $y_n\in qBq$ for $n\ge 1$. From \eqref{gpeq3.10} it follows that $\eta\in L^2(qBq)$. For $b\in B$,
\begin{align}
 tqb\xi &= tJb^*qJ\xi = Jb^*qJt\xi\notag\\
\label{gpeq3.12}
&= \lim_{n\to\infty} Jb^*qJy_n\xi = \lim_{n\to\infty} y_nqb\xi,
\end{align}
showing that $tqb\xi\in L^2(qB)$. Thus $L^2(qB)$ is an invariant subspace for $t$. The projection onto it is $qe_B$, so $tqe_B = qe_Btqe_B$. Since $t$ is self-adjoint, we obtain that $qe_B$ commutes with $(vBv^*)'\cap (q\langle M,e_B\rangle q)$, establishing centrality.

It was established in equation \eqref{gpeq3.6} that $vxv^*\in (vBv^*)'\cap (q\langle M,e_B\rangle q)$ whenever $x\in (B'\cap \langle M,e_B\rangle)p$, and so each such $vxv^*$ commutes with $qe_B$. Thus
\begin{align}
 v^*e_Bvx &= v^*qe_Bvxv^*v = v^*vxv^*qe_Bv\notag\\
\label{gpeq3.13}
&= xv^*e_Bv
\end{align}
for $x\in (B'\cap \langle M,e_B\rangle)p$, showing that $v^*e_Bv$ is central in $(B'\cap \langle M,e_B\rangle)p$.
\end{proof}

Since the centrality of the projections $v^*e_Bv$ in $(B'\cap \langle M,e_B\rangle)v^*v$ is crucial to our subsequent arguments, the following corollary highlights why the hypothesis $B'\cap M\subseteq B$ is essential.
\begin{cor}\label{gpcor3.3}
Let $B$ be a von Neumann subalgebra of a finite von Neumann algebra $M$. Then $e_B$ is central in $B'\cap \langle M,e_B\rangle$ if and only if $B'\cap M\subseteq B$.
\end{cor}

\begin{proof}
If $B'\cap M\subseteq B$, then centrality of $e_B$ is a special case of Lemma \ref{gplem3.2} (i) with $v=1$. Conversely, suppose that $e_B$ is central in $B'\cap \langle M,e_B\rangle$ and consider $x\in B'\cap M$. Then $x$ commutes with $e_B$ so $x\in B$ by Proposition \ref{ListOfFacts} (iii). Thus $B'\cap M\subseteq B$.
\end{proof}

\begin{lem}\label{gplem3.6}
Let $B$ be a von Neumann subalgebra of a finite von Neumann algebra $M$, suppose that $B'\cap M\subseteq B$, and let $v\in \oneGN (B)$. Then any subprojection of $v^*e_Bv$ in $(B'\cap \langle M,e_B\rangle)v^*v$ has the form $pv^*e_Bv$ where $p$ is a central projection in $B$.
\end{lem}

\begin{proof}
By Lemma \ref{gplem3.1} (i), $v^*e_Bv\in (B'\cap \langle M,e_B\rangle)v^*v$. Suppose that a projection $q\in (B'\cap \langle M,e_B\rangle)v^*v$ lies below $v^*e_Bv$. Then $(vqv^*)^2 = vqv^*vqv^* = vqv^*$, so $vqv^*$ is a projection. The relation
\begin{equation}\label{gpeq3.27}
 (vqv^*)(vbv^*) = vq(v^*vbv^*v)v^* = v(v^*vbv^*v)qv^*=(vbv^*)(vqv^*),\qquad b\in B,
\end{equation}
shows that $vqv^*\in (vBv^*)'\cap vv^*\langle M,e_B\rangle vv^*$. Moreover, $vqv^*\le vv^*e_Bvv^* = e_Bvv^*$, so there exists a projection $f\in vv^*Bvv^*$ such that $vqv^* = fe_B$. For $b\in B$,
\begin{equation}\label{gpeq3.28}
fvbv^*e_B= fe_Bvbv^* = vqv^*vbv^* = vbv^*vqv^*= vbv^*fe_B,
\end{equation}
and so $fvbv^* = vbv^*f$. Thus 
\begin{equation}
f\in (vBv^*)'\cap vv^*Bvv^*.
\end{equation}
If $b_0\in B$ is such that $vv^*b_0vv^*$ commutes with
$vBv^*$, then
\begin{equation}
vv^*b_0vv^*vbv^*=vbv^*vv^*b_0vv^*,\ \ b\in B.
\end{equation}
Multiply on the left by $v^*$ and on the right by $v$ to
obtain
\begin{equation}
v^*b_0vv^*vbv^*v=v^*vbv^*vv^*b_0v,\ \ b\in B.
\end{equation}
Thus
\begin{equation}
v^*b_0v\in (v^*vBv^*v)'\cap v^*vMv^*v=v^*v(B'\cap M)v^*v
=v^*vZ(B)v^*v.
\end{equation}
Consequently, $vv^*b_0vv^*\in vZ(B)v^*$. It follows that
$(vBv^*)'\cap vv^*Bvv^*\subseteq vZ(B)v^*$,
and so there is a central projection $p\in Z(B)$ so that $f=vpv^*$. We now have
\begin{equation}\label{gpeq3.30}
 q = v^*fe_Bv = v^*vpv^*e_Bv = pv^*e_Bv,
\end{equation}
as required.
\end{proof}

In Section \ref{gpsec4}, we wish to use the projections $v^*e_Bv$ to investigate an intertwiner $v$ of a tensor product $B=B_1\vnotimes B_2\subset M_1\vnotimes M_2=M$, where each $B_i'\cap M_i\subseteq B_i$.  In conjunction with Proposition \ref{ListOfFacts} (xiii), Tomita's commutation theorem gives
\begin{equation}
(B'\cap\langle M,e_B\rangle)\cong (B_1'\cap\langle M_1,e_{B_1}\rangle)\vnotimes(B_2'\cap\langle M_2,e_{B_2}\rangle).
\end{equation}
By Lemma \ref{gplem3.2}, such an intertwiner gives rise to a central projection $v^*e_Bv$ in $(B'\cap\langle M,e_B\rangle)v^*v$. Unfortunately, in general the projection $v^*v$ will not factorise as an elementary tensor of projections $b_1\otimes b_2$, with $b_i\in B_i$, and so the algebra $(B'\cap\langle M,e_B\rangle)v^*v$ will not decompose as a tensor product. This prevents us from applying tensor product techniques to the projection $v^*e_Bv$ directly. However, standard von Neumann algebra theory (see for example \cite{KR.1}) gives a central projection $P\in B'\cap\langle M,e_B\rangle$ such that $v^*e_Bv = Pv^*v$. Since we need to ensure that the projection $P$ fully reflects the properties of $v$, we cannot just appeal to the general theory to obtain $P$ so we give an explicit construction in Definition \ref{DefPv} below.  The subsequent lemmas set out the properties of $P$ that we require later.

\begin{defn}\label{DefPv}
Let $B\subset M$ be an inclusion of finite von Neumann algebras with $B'\cap M\subseteq B$ and let $v\in\oneGN_M(B)$. Let $z\in Z(B)$ be the central support of $v^*v$. Define $p_0$ to be $v^*v$, and let $\{p_0,p_1,\ldots\}$ be a family of nonzero pairwise orthogonal projections in $B$ which is maximal with respect to the requirements that $p_n\le z$ and each $p_n$ is equivalent in $B$ to a subprojection in $B$ of $p_0$. Since  two projections in a von Neumann algebra with non-orthogonal central supports have equivalent non-zero subprojections, maximality gives $\sum\limits_{n\ge 0} p_n=z$. For $n\ge1$, choose partial isometries $w_n\in B$ so that $w^*_nw_n =q_n \le p_0$ and $w_nw^*_n = p_n$. Then define $v_n=vw^*_n\in \oneGN (B)$. Lemma \ref{gplem3.1} shows that $v^*_ne_Bv_n\in (B'\cap \langle M,e_B\rangle)v^*_nv_n$ and this space is $(B'\cap \langle M,e_B\rangle)p_n$ since $v^*_nv_n = w_nv^*vw^*_n = p_n$. In particular, $\{v^*_ne_Bv_n\}_{n\ge 0}$ is a set of pairwise orthogonal projections so we may define a projection $P_v=\sum\limits_{n\ge 0}v^*_ne_Bv_n$ in $\langle M,e_B\rangle$. 
\end{defn}

\begin{lem}\label{gplem4.3}
With the notation of Definition \ref{DefPv}, the projection $P_v$ is central in $B'\cap \langle M,e_B\rangle$ and satisfies $P_vv^*v = v^*e_Bv$.
\end{lem}

\begin{proof}
 This projection is $P_v = \sum\limits_{n\ge 0} w_nv^*e_Bvw^*_n$. By Lemma \ref{gplem3.2} there exists $t\in Z(B'\cap \langle M,e_B\rangle)$ so that $v^*e_Bv = tv^*v$, and so $P_v$ becomes
\begin{equation}\label{gpeq4.22}
 P_v = \sum_{n\ge 0} w_ntv^*vw^*_n= \sum_{n\ge 0} tw_nv^*vw^*_n = \sum_{n\ge 0} tw_np_0w^*_n = tz.
\end{equation}
Thus $P_v\in B'\cap \langle M,e_B\rangle$ and $P_vv^*v = tv^*v = v^*e_Bv$ since $z$ is the central support  of $v^*v$. Since $z\in Z(B)\subset Z(B'\cap \langle M,e_B\rangle)$, it follows that $P_v=tz$ also lies in $Z(B'\cap\langle M,e_B\rangle)$
\end{proof}

\begin{rem}\label{gprem3.7}
The proof of Lemma \ref{gplem4.3} shows that $P_v$ is the minimal projection in $B'\cap\langle M,e_B\rangle$ with $P_vv^*v=v^*e_Bv$. This gives a canonical description of $P_v$ which is independent of the choices made in Definition \ref{DefPv}. The explicit formulation of the definition is useful in transfering properties from $v^*e_Bv$ to $P_v$.
\end{rem}
We now identify the subprojections of $P_v$. This will be accomplished by the next lemma, which considers a wider class of projections needed subsequently. Let $\{v_i\}^\infty_{i=1}$ be a sequence from $\oneGN (B)$ satisfying $v_iv^*_j=0$ for $i\ne j$, let $p\in B$ be the projection $\sum\limits^\infty_{i=1}v^*_iv_i$ and let $P\in \langle M,e_B\rangle$ be the projection $\sum\limits^\infty_{i=1} v^*_ie_Bv_i$. In particular, the projection $P_v$ of Definition \ref{DefPv} is of this form. Let $N(P)$ denote the von Neumann algebra
\begin{equation}\label{DefNP}
N(P) = \{x\in pMp\colon \ xP=Px\}\subseteq M.
\end{equation}

\begin{lem}\label{gplem4.5}
Let $P = \sum\limits^\infty_{i=1} v^*_ie_Bv_i$ be as above, and let $p = \sum\limits^\infty_{i=1} v^*_iv_i\in B$.
\begin{itemize}
 \item[\rm (i)] If $x\in pMp$ satisfies $xP=0$, then $x=0$; 
\item[\rm (ii)] The map $x\to xP$ is a $*$-isomorphism of $N(P)$ into $\langle M,e_B\rangle$;
\item[\rm (iii)] A projection $Q\in \langle M,e_B\rangle$ satisfies $Q\le P$ if and only if there exists a projection $f\in N(P)$ such that $Q=fP$. Moreover, if $P$ and $Q$ lie in $B'\cap \langle M,e_B\rangle$, then $f\in Z(B)$ and $Q$ has the same form as $P$.
\end{itemize}
\end{lem}

\begin{proof}
 (i)~~Suppose that $x\in pMp$ and $xP=0$. Then
\begin{equation}\label{gpeq4.26}
 x \sum^\infty_{i=1} v^*_ie_Bv_i = 0.
\end{equation}
Multiply on the right in \eqref{gpeq4.26} by $v^*_kv_k$ to obtain $xv^*_ke_Bv_k = 0$ for $k\ge 1$. The pull down map gives $xv^*_kv_k=0$. Summing over $k$ shows that $xp=0$ and the result follows since $x=xp$.

\n (ii)~~Since $P\in N(P)'$, the map $x\mapsto xP$ is a $*$-homomorphism on $N(P)$. It has trivial kernel, by (i), so is a $*$-isomorphism.

\n (iii)~~If $f\in N(P)$, then it is clear that $fP$ is a projection below $P$, since $f$ commutes with $P$. Conversely, consider a projection $Q\le P$ with $Q\in \langle M,e_B\rangle$. The introduction of partial sums below is to circumvent some questions of convergence.

Define $P_k = \sum\limits^k_{i=1} v^*_ie_Bv_i$. Then $\lim\limits_{k\to\infty} P_k=P$ strongly, so $P_kQP_k$ converges strongly to $PQP=Q$. For $m,n\ge 1$, let $b_{m,n}\in B$ be the element such that $b_{m,n}e_B = e_Bv_mQv^*_ne_B$. Then
\begin{align}
 P_kQP_k &= \sum^k_{m,n=1} v^*_me_Bv_mQv^*_ne_Bv_n
= \sum^k_{m,n=1} v^*_mb_{m,n}e_Bv_n\notag\\
&= \sum^k_{m,n=1} v^*_mb_{m,n}e_Bv_nv^*_nv_n
\label{gpeq4.27}
= \sum^k_{m,n=1} v^*_mb_{m,n}v_nv^*_ne_Bv_n.
\end{align}
Now define $x_k\in pW^*(\oneGN (B))p$ by $x_k = \sum\limits^k_{m,n=1} v^*_mb_{m,n}v_n$. The relations $v_iv^*_j = 0$ for $i\ne j$ allow us to verify that $x_kP = P_kQP_k$, and consequently
\begin{equation}\label{gpeq4.28}
 x_kP = Px^*_k, \qquad k\ge 1,
\end{equation}
since $P_kQP_k$ is self-adjoint. Thus
\begin{equation}\label{gpeq4.29}
 \sum^\infty_{i=1} x_kv^*_ie_Bv_i = \sum^\infty_{i=1} v^*_ie_Bv_ix^*_k.
\end{equation}
The sums in \eqref{gpeq4.29} converge in $\|\cdot\|_1$--norm, so we may apply the pull down map to obtain $x_kp = px^*_k$. Since $x_k=x_kp$, we conclude that $x_k$ is self-adjoint. Thus, from \eqref{gpeq4.28}, $x_k$ commutes with $P$, and so lies in $N(P)$. From above, $x_kP\ge 0$ and $\|x_kP\|\le 1$, so $x_k\ge 0$ and $\|x_k\|\le 1$ by (ii). Let $f$ be a $\sigma$-weak accumulation point of the sequence $\{x_k\}^\infty_{k=1}$. Then $f\ge 0$, $\|f\|\le 1$, and $f\in N(P)$. Since $x_kP = P_kQP_k$, we conclude that $fP=Q$. It now follows from (ii) that $f$ is a projection in $N(P)$. 

If $P\in B'\cap \langle M,e_B\rangle$, then the pull down map gives $bp=pb$ for $b\in B$, so $p\in Z(B)$. If also $Q\in B'\cap \langle M,e_B\rangle$, then $Q=fP$ and $f\le p$. If $b\in Bp$ then commutation with $Q$ gives $(bf-fb)P=0$, so $f\in B'\cap M=Z(B)$, by (i). Finally $Q=fP=\sum_i(v_if)^*e_B(v_if)$, so is of the same form as $P$.
\end{proof}

\section{Intertwiners of tensor products}\label{gpsec4}

\indent 

In this section we will prove one of our main results, the equality of $W^*(\oneGN (B_1))\ovl\otimes\break W^*(\oneGN (B_2))$ and $W^*(\oneGN (B_1\ovl\otimes B_2))$, where $B_i \subseteq M_i$, $i=1,2$, are inclusions of finite von Neumann algebras satisfying $B'_i\cap M_i\subseteq B_i$. The key theorem for achieving this is the following one, which enables us to detect those central projections in corners of the relative commutant of the basic construction which arise from intertwiners. It is inspired by \cite[Proposition 2.7]{Chifan.Normalisers},
  although is not a direct generalization of that result. For comparison,
\cite[Proposition 2.7]{Chifan.Normalisers} shows that, in the case of a masa $A$, a projection $P\in A'\cap \langle M,e_A\rangle$ which is subequivalent to $e_A$ dominates an operator $v^*e_Bv$ for some $v\in \twoGN (A)$. Example \ref{gpexm4.2} below will show that such a result will not hold in general without additional hypotheses.

\begin{thm}\label{gpthm4.1}
Let $A$ be an abelian von Neumann algebra with a fixed faithful normal semifinite weight $\Phi$. Let $B$ be a von Neumann subalgebra of a finite von Neumann algebra $M$ with a faithful normal trace $\tau$ satisfying $B'\cap M\subseteq B$. Fix a projection $q\in A\vnotimes B$ and suppose that $P\in (A\vnotimes (B'\cap \langle M,e_B\rangle))q$ is a nonzero projection such that $P \precsim (1\otimes e_B)$ in $A\vnotimes \langle M,e_B\rangle$, and satisfies
\begin{equation}\label{gpeq4.1}
(\Phi\otimes\Tr)(Pr) \le (\Phi\otimes\tau)(qr)
\end{equation}
for all projections $r\in Z(q(A\vnotimes B)q)$. Then there exists an element $v\in \oneGN _{A\vnotimes M}(A\vnotimes B)$ such that $P=v^*(1\otimes e_B)v$.
\end{thm}

Before embarking on the proof, let us recall that, for a finite von Neumann algebra $M$ with a faithful normal trace $\tau$, we regard the Hilbert space $L^2(M)$ as the completion of $M$ in the norm $\|x\|_{2,\tau}=\tau(x^*x)^{1/2}$ and $L^1(M)$ as the completion of $M$ in the norm $\|x\|_{1,\tau}=\tau(|x|)$. The Cauchy-Schwarz inequality gives $\|xy^*\|_{1,\tau}\leq\|x\|_{2,\tau}\|y\|_{2,\tau}$ for $x,y\in M$ and so this inequality allows us to define $\zeta\eta^*\in L^1(M)$ for $\zeta,\eta\in L^2(M)$. In particular, if $(y_n)$ is a sequence in $M$ converging to $\eta\in L^2(M)$, then $y_n^*y_n\rightarrow \eta^*\eta$ in $L^1(M)$.  

Recall too that we can regard elements of $L^2(M)$ as unbounded operators on $L^2(M)$ affiliated to $M$. The only fact we need about these unbounded operators is that if $\eta\in L^2(M)$ satisfies $\eta^*\eta\in M$ (regarded as a subset of $L^1(M)$), then in fact $\eta\in M$. This follows as $\eta$ has a polar decomposition $v(\eta^*\eta)^{1/2}$, where $v$ is a partial isometry in $M$ and $(\eta^*\eta)^{1/2}$ is an element of $L^2(M)$, which lies in $M$ if $\eta^*\eta$ does.

\begin{proof}[Proof of Theorem \ref{gpthm4.1}.]
 The first case that we will consider is where $A={\bb C}$ and $\Phi$ is the identity map. Then the hypothesis becomes
\begin{equation}\label{gpeq4.2}
\Tr(Pr) \le\tau(qr)
\end{equation}
for all projections $r\in Z(B)q$. Since $P\precsim e_B$, there exists a partial isometry $V\in \langle M,e_B\rangle$ such that $P=V^*V$ and $VV^*\le e_B$. Define the map $\theta\colon \ qBq\to Be_B$ by
\begin{equation}\label{gpeq4.3}
\theta(qbq) = VqbqV^* = e_BVqbqV^*e_B,\qquad b\in B.
\end{equation}
Then $\theta$ is a $*$-homomorphism since $V^*V$ commutes with $qBq$, and so there is a $*$-homomorphism $\phi\colon \ qBq\to B$ so that $\theta(qBq) = \phi(qbq)e_B$ for $qbq\in qBq$. Thus
\begin{equation}\label{gpeq4.4}
 qbqV^* = qbqV^*VV^* = V^*VqbqV^*= V^*e_B\phi(qbq) = V^*\phi(qbq)
\end{equation}
for $qbq\in qBq$. Now define $\eta\in L^2(M)$ by $\eta = JV^*\xi$, and observe that $\eta\ne 0$ since $VJ\eta = VV^*\xi = b_0\xi$, where $VV^* = b_0e_B$ for some $b_0\in B$. If we apply \eqref{gpeq4.4} to $\xi$, then the result is
\begin{equation}\label{gpeq4.5}
 qbqJ\eta = V^*\phi(qbq)\xi = V^*J\phi(qb^*q)J\xi= J\phi(qb^*q)JV^*\xi= J\phi(qb^*q)\eta,\qquad qbq\in qBq,
\end{equation}
where we have used $\langle M,e_B\rangle = (JBJ)'$ to commute $J\phi(qb^*q)J$ with $V^*$. Multiply \eqref{gpeq4.10} on the left by $J$ and replace $b$ by $b^*$ to obtain
\begin{equation}\label{gpeq4.6}
 \eta qbq = \phi(qbq)\eta,\qquad qbq\in qBq.
\end{equation}
Taking $b=1$, \eqref{gpeq4.5} becomes
\begin{equation}\label{gpeq4.7}
 JqJ\eta = \phi(q)\eta.
\end{equation}
 so \begin{equation}\label{gpeq4.8}
 JqV^*\xi = \phi(q)\eta.
\end{equation}
Multiply on the left by $VJ$ to obtain
\begin{equation}\label{gpeq4.9}
 \phi(q)\xi = VJ\phi(q)\eta,
\end{equation}
showing that $\phi(q)\eta \ne 0$. 
From \eqref{gpeq4.7}, $\phi(q)\eta q = \phi(q)\eta\ne 0$, and this allows us to assume in \eqref{gpeq4.6} that the vector $\eta$ is nonzero and satisfies $\eta q=\eta$, $\phi(q)\eta = \eta$ by replacing $\eta$ with $\phi(q)\eta q$ if necessary. For unitaries $u\in qBq$, \eqref{gpeq4.6} becomes 
\begin{equation}\label{gpeq4.10}
 \phi(u^*)\eta u = \phi(u^*) \phi(u)\eta = \phi(q)\eta = \eta.
\end{equation}

Choose a sequence $\{x_n\}^\infty_{n=1}$ from $M$ such that $x_n\xi\to \eta$ in $\|\cdot\|_{2,\tau}$-norm. Since $\phi(q)\eta q = \eta$, we may assume that $\phi(q)x_nq = x_n$ for $n\ge 1$. Let $y_n$ be the element of minimal $\|\cdot\|_{2,\tau}$-norm in $\ovl{\rm conv}^w\{\phi(u^*)x_nu\colon \ u\in {\cl U}(qbq)\}$. Since
\begin{equation}\label{gpeq4.11}
 \|\phi(u^*)x_nu\xi - \eta\|_{2,\tau} = \|\phi(u^*)(x_n\xi-\eta)u\|_{2,\tau}\le \|x_n\xi-\eta\|_{2,\tau}
\end{equation}
for all $u\in {\cl U}(qBq)$, we see that $\|y_n\xi-\eta\|_{2,\tau} \le \|x_n\xi-\eta\|_{2,\tau}$, so $y_n\xi\to \eta$ in $\|\cdot\|_{2,\tau}$-norm and $\phi(u^*)y_nu = y_n$ for $n\ge 1$ by the choice of $y_n$. Then $y_nu = \phi(u)y_n$ for $u\in {\cl U}(qBq)$, so
\begin{equation}\label{gpeq4.12}
 y_nqbq = \phi(qbq)y_n, \qquad n\ge 1,\quad qbq\in qBq.
\end{equation}
Thus
\begin{equation}\label{gpeq4.13}
 y^*_ny_nqbq = y^*_n\phi(qbq)y_n,\qquad n\ge 1,\quad qbq\in qBq,
\end{equation}
and this implies that $y^*_ny_n\in(qBq)'\cap qMq = (B'\cap M)q = Z(B)q$ for each $n\ge 1$. The discussion preceding the proof ensures that $y^*_ny_n\to \eta^*\eta$ in $L^1(M)$, so we see that $\eta^*\eta \in L^1(Z(B)q)$. For each $z\in Z(B)q$,
\begin{align}
|\tau(\eta^*\eta zq)| &= \lim_{n\to\infty} |\tau(y^*_ny_nqz)|
= \lim_{n\to\infty} |\langle y_nqz\xi, y_n\xi\rangle|\notag\\
&= |\langle Jz^*qJ\eta,\eta\rangle|
= |\langle Jz^*qJJV^*\xi, JV^*\xi\rangle|\notag\\
&= |\langle z^*qV^*\xi, V^*\xi\rangle|= |\langle Vz^*qV^*\xi,\xi\rangle|\notag\\
&= |\langle \phi(z^*q)\xi,\xi\rangle|\
\label{gpeq4.14}
= |\tau(\phi(zq))|.
\end{align}

Now
\begin{align}
\Tr(Pzq) &=\Tr(V^*Vzq) =\Tr(VzqV^*)\notag\\
\label{gpeq4.15}
&= \Tr(\phi(zq)e_B) = \tau(\phi(zq)).
\end{align}
Thus, from \eqref{gpeq4.14}, \eqref{gpeq4.15} and the hypothesis \eqref{gpeq4.2},
\begin{equation}\label{gpeq4.16}
 |\tau(\eta^*\eta r)| = |\Tr(Pr)|\le \tau(r)
\end{equation}
for all projections $r\in Z(B)q$. Since $Z(B)q$ is abelian, simple measure theory allows us to conclude from \eqref{gpeq4.16} that $\eta^*\eta\in Z(B)q$ (rather than just $L^1(Z(B)q))$ and so $\eta \in M$, by the discussion prior to the start of the proof. Moreover, \eqref{gpeq4.16} also gives $\|\eta\|\le 1$, by taking $r$ to be the spectral projection of $\eta^*\eta$ for the interval $(c,\infty)$ where $c>1$ is arbitrary.

Since $\eta\in L^2(M)$ has been proved to lie in $M$, we rename this nonzero operator as $x\in M$. From above, $\|x\|\le 1$ and $x^*x\in Z(B)q$. Since $JV^*\xi = \eta = x\xi=xJ\xi$,  for $y\in M$ and $b\in B$,
\begin{align}
 \langle (V-x)y\xi,b\xi\rangle &= \langle y\xi, V^*Jb^*J\xi\rangle - \langle xy\xi,b\xi\rangle\notag\\
&= \langle y\xi, Jb^*JV^*\xi\rangle - \langle xy\xi,b\xi\rangle\notag\\
&= \langle y\xi, Jb^*xJ\xi\rangle - \langle xy\xi,b\xi\rangle\notag\\
&= \langle y\xi, x^*b\xi\rangle - \langle xy\xi, b\xi\rangle
\label{gpeq4.17}
= 0,
\end{align}
and so $e_BV = e_Bx$, implying that $V=e_Bx$. Thus $P=V^*V = x^*e_Bx$. Since $P^2=P$,
\begin{equation}\label{gpeq4.18}
 x^*e_Bx = x^*e_Bxx^*e_Bx = x^*{\bb E}_B(xx^*)e_Bx,
\end{equation}
so the pull down map gives
\begin{equation}\label{gpeq4.19}
 x^*x = x^*{\bb E}_B(xx^*)x.
\end{equation}
This equation is $x^*(1-{\bb E}_B(xx^*))x = 0$, so $(1-{\bb E}_B(xx^*))^{1/2}x = 0$. Thus ${\bb E}_B(xx^*)x=x$. If we multiply on the right by $x^*$ and apply ${\bb E}_B$, then we conclude that ${\bb E}_B(xx^*)$ is a projection. Moreover,
\begin{equation}\label{gpeq4.20}
 {\bb E}_B(xx^*)xx^* = xx^*
\end{equation}
so ${\bb E}_B(xx^*)\ge xx^*$ since $\|x\|\le 1$. The trace then gives equality, and so $x$ is a partial isometry with $x^*x$, $xx^*\in B$. Since $x=xq$ and $x^*e_Bx = P$, which commutes with $qBq$, we obtain
\begin{equation}\label{gpeq4.21}
xbx^*e_B = xqbqx^*xx^*e_B= xqbqx^*e_Bxx^*= xx^*e_Bxqbqx^*= e_Bxbx^*,\qquad b\in B,
\end{equation}
showing that $xBx^*\subseteq B$. Thus $x\in \oneGN (B)$ and $P=x^*e_Bx$. This completes the proof when $A=\mathbb C$ and $\Phi$ is the identity map.

The second case is when $A$ is an arbitrary abelian von Neumann algebra, and $\Phi$ is bounded, so we may assume that $\Phi$ is a state on $A$ since \eqref{gpeq4.1} is unaffected by scaling. The result now follows from the first case by replacing the inclusion $B\subseteq M$ by $A\vnotimes B \subseteq A\vnotimes M$. The trace on $A\vnotimes M$ is $\Phi\otimes\tau$ and $e_{A\vnotimes B}$ is $1\otimes e_B$, so the canonical trace on $\langle A\vnotimes M, e_{A\vnotimes B}\rangle$ is $\Phi \otimes Tr$.

The last case is where $\Phi$ is a faithful normal semifinite weight. Then there is a family $\{f_\lambda\}_{\lambda\in\Lambda}$ of orthogonal projections in $A$ with sum 1 such that $\Phi(f_\lambda) < \infty$ for each $\lambda\in\Lambda$. If $\Phi_\lambda$ denotes the restriction of $\Phi$ to $Af_\lambda$, then $\Phi_\lambda$ is bounded. If we replace $P,A,\Phi$ and $q$ by respectively $P(f_\lambda\otimes 1)$, $Af_\lambda$, $\Phi_\lambda$ and $q(f_\lambda\otimes 1)$, then we are in the second case. Thus there exists, for each $\lambda\in\Lambda$, a partial isometry $v_\lambda\in \oneGN _{Af_\lambda\vnotimes M}(Af_\lambda \vnotimes B)$ so that $P(f_\lambda\otimes 1) = v^*_\lambda (e_{Af_\lambda\vnotimes B})v_\lambda$. The central support of $v_\lambda$ lies below $f_\lambda\otimes 1$ so we may define $v\in \oneGN _{A\vnotimes M}(A\vnotimes B)$ by $v = \sum\limits_{\lambda\in\Lambda} v_\lambda$, and it is routine to check that $P = v^*e_Bv$.
\end{proof}

Theorem \ref{gpthm4.1} characterizes those projections in the basic construction which arise from intertwiners.
\begin{cor}
Given an inclusion $B\subseteq M$ of finite von Neumann algebras with $B'\cap M\subseteq B$ and a projection $q\in B$, a projection $P\in (B'\cap \langle M,e_B\rangle)q$ is of the form $v^*e_Bv$ for some intertwiner $v\in\oneGN(B)$ if and only if $P\precsim e_B$ in $\langle M,e_B\rangle$ and $\Tr(Pr)\leq\tau(qr)$ for all projections $r\in Z(B)q$. Furthermore in this case the domain projection $v^*v$ must lie in $Z(B)q$.
\end{cor}
\begin{proof}
Taking $A=\mathbb C$ and $\Phi$ to be the identity in theorem \ref{gpthm4.1} shows that any projection $P$ satisfying the conditions of the corollary is of the form $v^*e_Bv$ for some intertwiner $v\in\oneGN(B)$. Lemma \ref{gplem3.1} (ii) then shows that $v^*v\in Z(B)q$. Conversely, given an intertwiner $v\in\oneGN(B)$, Lemma \ref{gplem3.1} (ii) shows that $v^*e_Bv\in (B'\cap\langle M,e_B\rangle)q$ precisely when $v^*v\in Z(B)q$. The other two conditions of the corollary follow as $v^*e_Bv\sim vv^*e_B\leq e_B$ in $\langle M,e_B\rangle$, and for a projection $r\in Z(B)q$, $\Tr(v^*e_Bvr)=\tau(v^*vr)\leq\tau(qr)$.
\end{proof}

The tracial hypothesis \eqref{gpeq4.1} of Theorem \ref{gpthm4.1} is an extra ingredient in this theorem as compared with \cite[Proposition 2.7]{Chifan.Normalisers}. The following example shows that Theorem \ref{gpthm4.1} can fail without this hypothesis. 

\begin{exm}\label{gpexm4.2}
Let $R$ be the hyperfinite $\text{II}_1$ factor and fix an outer automorphism $\theta$ of period two. Let $M = {\bb M}_3\otimes R$ and let
\[
B = \left\{\begin{pmatrix}
            a&0&0\\ 0&b&c\\ 0&\theta(c)&\theta(b)
           \end{pmatrix}\colon\ a,b,c\in R\right\} \subseteq M.
\]
Note that $B\cong R\oplus (R\rtimes_\theta {\bb Z}_2)$. It is straightforward to verify that $B'\cap M = Z(B) = {\bb C}e_{11} \oplus {\bb C}(e_{22}+e_{33})$ where $\{e_{i,j}\}^3_{i,j=1}$ are the matrix units. Let $P \in \langle M,e_B\rangle$ be $(\sqrt 2\ e_{12}) e_B(\sqrt 2\ e_{21})$. Since ${\bb E}_B(e_{22}) = (e_{22}+e_{33})/2$, $P$ is a projection, and it is routine to verify that $P$ commutes with $B$. If there is a nonzero intertwiner $v\in M$ such that $v^*e_Bv\leq
P$, then $v=ve_{11}$. Direct calculation shows that $v$
would then have the form $we_{11}$ for some partial isometry $w\in
R$, so $v^*e_Bv$ would be $qe_{11}e_B$ for some projection $q\in R$. However, this nonzero
projection is orthogonal to $P$ and so cannot lie under it. Thus the conclusion of Theorem \ref{gpthm4.1} fails in this case. Note that $\Tr(Pe_{11}) = \tau(2e_{11}) = 2/3$, while $\tau(e_{11}) = 1/3$, so the tracial hypothesis of Theorem \ref{gpthm4.1} is not satisfied.

It is worth noting 
 that $B'\cap \langle M,e_B\rangle$ can be explicitly calculated in this case. This algebra is abelian and five dimensional with minimal projections $e_{11}e_B$, $(1-e_{11})e_B$, $(1-e_{11})(1-e_B)$, $(\sqrt 2\ e_{12}) e_B(\sqrt 2\ e_{21})$, and $e_{21}e_Be_{12}+e_{31}e_Be_{13}$. The corresponding $B$--bimodules in $L^2(M)$ are  generated respectively by the  vectors $e_{11} $, $e_{22}+e_{33} $, $e_{22}-e_{33} $, $e_{12}$, and  
$e_{21}$. $\hfill\square$
\end{exm}

For the remainder of the section we fix inclusions $B_i\subseteq M_i$ of finite von Neumann algebras satisfying $B'_i\cap M_i\subseteq B_i$ for $i=1,2$, and we denote the inclusion $B_1\vnotimes B_2 \subseteq M_1\vnotimes M_2$ by $B\subseteq M$. 
 For $i=1,2$, let \[S_i = \sup\{P_j\in Z(B'_i \cap \langle M_i,e_{B_i}\rangle)\colon \Tr_i(P_j) < \infty\}.\] Note that $S_i$ acts as the identity on $L^2(Z(B'_i\cap \langle M_i,e_{B_i}\rangle),\Tr_i)$. Given $v\in\oneGN_M(B)$, the projection $P_v$ of Definition \ref{DefPv} satisfies $\Tr(P_v)\le 1$, and $P_v\in Z(B'\cap \langle M,e_B\rangle)$ by Lemma \ref{gplem4.3}. It follows from Lemma \ref{gplem3.7} that $P_v\le S_1\otimes S_2$. Although $\Tr_i$ restricted to $Z(B'_i \cap \langle M_i,e_{B_i}\rangle)$ might not be semifinite, it does have this property on the abelian von Neumann algebra $A_i = Z(B'_i\cap \langle M_i,e_{B_i}\rangle)S_i$ by the choice of $S_i$. Moreover, each $P_v$ is an element of $A_1\vnotimes \langle M_2,e_{B_2}\rangle$. We need two further projections which we define below.

\begin{defn}\label{DefQ}
For $i=1,2$, let $\mathcal P_i$ denote the collection of projections $R\in Z(B'_i \cap \langle M_i, e_{B_i}\rangle)$ which are expressible as
\[
 R = \sum_{n\ge 1} v^*_ne_{B_i}v_n,\qquad v_n\in \oneGN (B_i)
\]
where $\{v^*_nv_n\}_{n\ge 1}$ is an orthogonal set of projections in $B_i$. Such a projection satisfies $\Tr_i(R)\le 1$. Let $Q_i$ be the supremum of the projections in $\mathcal P_i$,  so $Q_i\le S_i$.$\hfill\square$ 
\end{defn}
Our next objective is to show that each projection $P_v$ arising from an intertwiner lies below $Q_1\otimes Q_2$. For the next two lemmas, let $v\in\oneGN_M(B)$ be fixed. We continue to employ the notation $P_v$ for the projection $\sum\limits_n v^*_ne_Bv_n \in Z(B'\cap \langle M,e_B\rangle)$ which satisfies $P_vv^*v = v^*e_Bv$. Let $A_1$ be the abelian von Neumann algebra $Z(B'_1\cap \langle M_1,e_{B_1}\rangle)S_1$ on which $\Phi$, the restriction of $\Tr_1$, is semifinite. 

\begin{lem}\label{gplem4.6}
 If $r$ is a projection in $A_1\vnotimes Z(B_2)$, then
\begin{equation}\label{gpeq4.30}
 (\Phi\otimes\Tr_2)(P_vr) \le (\Phi\otimes\tau_2)(r).
\end{equation}
\end{lem}

\begin{proof}
There is a measure space $(\Omega,\Sigma,\mu)$ so that $A_1$ corresponds to $L^\infty(\Omega)$ while $\Phi$ is given by integration with respect to the $\sigma$-finite measure $\mu$. Then $P_v$ is viewed as a projection-valued function $P_v(\omega)$, with the same representation $r(\omega)$ for $r$. For $i=1,2$, let $\Psi_i$ be the pull down map for $\langle M_i,e_{B_i}\rangle$. Then $(\Psi_1\otimes \Psi_2)(P_v) = \sum\limits_n v^*_nv_n$ which is a projection, so has norm 1. By Lemma \ref{gplem2.5}, $\|(I\otimes \Psi_2)(P_v)\|\le 1$, and so this element of $A_1\vnotimes M_2$ can be represented as a function $f(\omega)$ with $\|f(\omega)\|\le 1$ almost everywhere. It follows that
\begin{align}
 (\Phi\otimes\tau_2)(r) &= \int_\Omega \tau_2(r(\omega))d\mu(\omega)
\ge \int_\Omega \tau_2(r(\omega)f(\omega)) d\mu(\omega)\notag\\
&= (\Phi\otimes\tau_2)((I\otimes\Psi_2)(P_v)r)
= (\Phi\otimes \tau_2) ((I\otimes\Psi_2)(P_vr))\notag\\
\label{gpeq4.31}
&= (\Phi\otimes\Tr_2)(P_vr),
\end{align}
where the penultimate equality is valid because $r\in A_1\vnotimes Z(B_2)$ and the $M_2$-bimodular property of $\Psi_2$ applies.
\end{proof}

\begin{lem}\label{gplem4.7}
 For $v\in \oneGN (B)$, the associated projection $P_v\in Z(B'\cap \langle M,e_B\rangle)$ satisfies $P_v\le Q_1\otimes Q_2$.
\end{lem}

\begin{proof}
The remarks preceding Definition \ref{DefQ} show that $P_v\leq S_1\otimes S_2$. We will prove the
stronger inequality $P_v\leq  
 S_1\otimes Q_2$, which is sufficient to establish the result since  we will also have $P_v\le Q_1\otimes S_2$ by a symmetric argument.

As before, let $A_1$ denote $Z(B'_1\cap \langle M_1,e_{B_1}\rangle)S_1$ and let $\Phi$ be the restriction of $\Tr_1$ to $A_1$. Then, as noted earlier, $A_1$  is an abelian von Neumann algebra and $\Phi$ is a faithful normal semifinite weight on $A_1$. Let $\{q_n\}^\infty_{n=1}$ be a maximal family of nonzero orthogonal projections in $Z(B_1)\vnotimes B_2$ so that $P_vq_n =w^*_n(1\otimes e_{B_2}) w_n$ for partial isometries $w_n\in A_1\vnotimes M_2$ which are intertwiners of $A_1\vnotimes B_2$. Let $q=\sum\limits^\infty_{n=1} q_n$, defining $q$ to be $0$ if no such projections exist.  We will first show that $P_v \le q$, so suppose that $(1-q)P_v\ne 0$.

The central support of $1\otimes e_{B_2}$ in $A_1\vnotimes \langle M_2,e_{B_2}\rangle$ is 1, so there is a nonzero subprojection $Q$ of $(1-q)P_v$ in $A_1\vnotimes \langle M_2,e_{B_2}\rangle$ with $Q\precsim 1 \otimes e_{B_2}$ in this algebra. The projection $P_v$ has the form $P_v = \sum\limits_{n\ge 0} v^*_ne_Bv_n$ where $\sum\limits_{n\ge 0} v^*_nv_n$ is the central support $z\in Z(B)$ of $v^*v\in B$. With this notation, (\ref{DefNP}) becomes
\[
 N(P_v) = \{x\in zMz\colon \ xP_v=P_vx\}.
\]
By Lemma \ref{gplem4.5} (iii), there is a projection $f\in N(P_v)$ so that $Q = fP_v$. Both $Q$ and $P_v$ commute with $B_1\otimes 1$, so the relation $(b_1\otimes 1)Q-Q(b_1\otimes 1) = 0$ for $b_1\in B_1$ becomes $((b_1\otimes 1)f - f(b_1\otimes 1))P_v = 0$. The element $(b_1\otimes 1)f-f(b_1\otimes 1)$ lies in $N(P_v)$ so, by Lemma \ref{gplem4.5} (i), $(b_1\otimes 1)f = f(b_1\otimes 1)$ for all $b_1\in B_1$. This shows that $f\in B'_1\cap N(P_v)$. Moreover, $f=(1-q)f$ follows from the equation
\begin{equation}\label{gpeq4.32}
 0 = qQ = qfP_v,
\end{equation}
which implies that $qf=0$ since $qf\in N(P_v)$. Note that $(1-q)z\ne 0$, otherwise $zf=0$. Now
\begin{equation}
 B'_1 \cap N(P_v) \subseteq B'_1\cap M = Z(B_1)\vnotimes M_2
\end{equation}
and
\begin{equation}
\Big(Z(B_1)\vnotimes B_2\Big)' \cap\Big(Z(B_1)\vnotimes M_2\Big) = Z(B_1) \vnotimes Z(B_2) \subseteq Z(B_1)\vnotimes B_2.
\end{equation}
Thus the inclusion
\begin{equation}
 (1-q)z(Z(B_1) \vnotimes B_2)z(1-q) \subseteq (1-q)z(B'_1\cap N(P_v)) z(1-q)
\end{equation}
has the property that the first algebra contains its relative commutant in the second algebra, which is the hypothesis of Lemma \ref{gplem4.4}. Thus we may choose a nonzero projection $b\in (1-q)z(Z(B_1)\vnotimes B_2)z(1-q)$ with $b\precsim f$ in $(1-q)z(B'_1\cap N(P_v)) z(1-q)$. Let $w$ be a partial isometry in this algebra with $w^*w=b$ and $ww^*\le f$, and note that $w$ commutes with $P_v$ by definition of $N(P_v)$. Then
\begin{equation}\label{gpeq4.33}
 bP_v = w^*wP_v = P_vw^*wP_v\sim wP_vw^* = ww^*P_v \le fP_v
\end{equation}
in $A_1\vnotimes \langle M_2,e_{B_2}\rangle$. Since $b\le z$, Lemma \ref{gplem4.5} (i) ensures that $bP_v\ne 0$. Moreover, $bP_v\precsim 1\otimes e_{B_2}$ in $A_1\vnotimes \langle M_2,e_{B_2}\rangle$ since $fP_v$ has this property.

Consider now a projection $r\in (A_1\vnotimes Z(B_2))b$. The inequality
\begin{equation}\label{gpeq4.34}
 \Phi \otimes \Tr_2(P_vr) \le \Phi\otimes\tau(r)
\end{equation}
is valid by Lemma \ref{gplem4.6}. Thus the hypotheses of Theorem \ref{gpthm4.1} are satisfied with $P$ replaced by $bP_v$. We conclude that there is an element $w\in \oneGN _{A_1\vnotimes M_2}(A_1\vnotimes B_2)$ so that $bP_v=w^*(1\otimes e_{B_2})w$. Since $b$ lies under $1-q$, this contradicts maximality of the $q_i$'s, proving that $P_vq=P_v$. Thus
\begin{equation}\label{neweq1}
P_v=\sum^{\infty}_{n=1}w_n^*(1\otimes e_{B_2})w_n,
\end{equation}
which we also write as $P_v=\sum\limits^{\infty}_{n=1}W_n^*W_n$ where $W_n$ is defined to be $(1\otimes e_{B_2})w_n\in A_1\vnotimes \langle M_2,e_{B_2}\rangle$.

As in Lemma \ref{gplem4.6}, we regard $(A_1,\Tr_1)$ as $L^{\infty}(\Omega)$ for a $\sigma$--finite measure space $(\Omega,\mu)$. We can then identify
$A_1\vnotimes \langle M_2,e_{B_2}\rangle$ with $L^{\infty}(\Omega,\langle M_2,e_{B_2}\rangle)$, and we write elements of this tensor product as uniformly bounded measurable functions on $\Omega$ with values in $\langle M_2,e_{B_2}\rangle$. Then
\begin{equation}\label{neweq2}
\Tr(P_v)=\int \sum^{\infty}_{n=1}\Tr_2(W_n(\omega)^*W_n(\omega))\,d\mu(\omega).
\end{equation}
Since $\Tr(P_v)<\infty$, we may neglect a countable number of null sets to conclude that
\begin{equation}\label{neweq3}
\Tr_2(P_v(\omega))=\Tr_2\left(\sum^{\infty}_{n=1}W_n(\omega)^*W_n(\omega)\right) <\infty,\ \ \ \omega\in \Omega,
\end{equation}
from which it follows that $P_v(\omega)\leq Q_2$ for $\omega\in \Omega$. Thus $P_v\leq 1\otimes Q_2$ which gives $P_v\leq S_1\otimes Q_2$, since the inequality $P_v\leq S_1\otimes S_2$ has already been established.
\end{proof}

We are now in a position to approximate an intertwiner in a tensor product.
\begin{thm}\label{gpthm4.8}
 Let $B_i\subseteq M_i$, $i=1,2$, be inclusions of finite von Neumann algebras satisfying $B'_i \cap M_i \subseteq B_i$. Given $v\in\oneGN_{M_1\vnotimes M_2}(B_1\vnotimes B_2)$ and $\vp>0$, there exist $x_1,\dots,x_k\in B_1\vnotimes B_2$, $w_{1,1},\dots,w_{1,k}\in\oneGN_{M_1}(B_1)$ and $w_{2,1},\dots,w_{2,k}\in\oneGN_{M_2}(B_2)$ such that:
 \begin{enumerate}
 \item $\|x_j\|\leq 1$ for each $j$;
\item\begin{equation}
\left\|v-\sum_{j=1}^kx_j(w_{1,j}\otimes w_{2,j})\right\|_{2,\tau}<\vp.
\end{equation}
\end{enumerate}
\end{thm}

\begin{proof}
Write $B\subseteq M$ for the inclusion $B_1\vnotimes B_2 \subseteq M_1\vnotimes M_2$ and fix $v\in \oneGN_M (B)$. Recall from Lemmas \ref{gplem4.3} and \ref{gplem4.7} that there is a projection $P_v\in Z(B'\cap \langle M,e_B\rangle)$ satisfying $P_vv^*v = v^*e_Bv$, and $P_v\le Q_1\otimes Q_2$ where $Q_i$ is the supremum of the set ${\cl P}_i$ of projections in $Z(B'_i\cap \langle M_i,e_{B_i}\rangle)$ specified in Definition \ref{DefQ}.  Thus $Q_i=\sum_kR_{i,k}$ for some countable sum of orthogonal projections $R_{i,k}\in\mathcal P_i$.  By Lemma \ref{gplem4.5} (iii), any subprojection of $R_{i,k}$ in $Z(B_i'\cap \langle M_i,e_{B_i}\rangle)$ is also in ${\cl P}_i$, so it follows that every subprojection of $Q_i$ in $B'\cap\langle M_i,e_{B_i}\rangle$ is also of the form $\sum_kR_{i,k}'$ for some countable sum of orthogonal projections $R_{i,k}'\in\mathcal P_i$.

The restriction $\Phi_i$ of $\Tr_i$ to $Z(B'_i\cap \langle M_i,e_{B_i}\rangle)Q_i$ is a normal semifinite weight on this abelian von Neumann algebra $A_i$, and $(A_i,\Tr_i)$ can be identified with $L^\infty(\Omega_i)$ for a $\sigma$-finite measure space $(\Omega_i,\Sigma_i,\mu_i)$. Since $P_v\le Q_1\otimes Q_2$, this operator can be viewed as an element of $L^\infty(\Omega_1\times \Omega_2, \mu_1\times\mu_2)$, and it also lies in the corresponding $L^2$-space since $\Tr(P_v)\le 1$. By Lemma \ref{gplem3.7} and the previous paragraph, $P_v$ can be approximated in $\|\cdot\|_{2,\Tr}$-norm by finite sums of orthogonal projections of the form $R_1\otimes R_2$, each lying in ${\cl P}_i$. These elementary tensors correspond to measurable rectangles in $\Omega_1\times \Omega_2$. By the definition of $\mathcal P_i$, each $R_i$ is close in $\|\cdot\|_{2,\Tr_i}$-norm to a finite sum $\sum\limits^k_{j=1} w^*_{i,j}e_{B_i}w_{i,j}$, with $w_{i,j}\in\oneGN_{M_i}(B_i)$.This allows us to make the following approximation:\ given $\vp>0$, there exist finite sets $\{w_{i,j}\}^k_{j=1} \in \oneGN_{M_i} (B_i)$, $i=1,2$, such that
\begin{equation}\label{gpeq4.36}
 \left\|P_v - \sum^k_{j=1} (w^*_{1,j}e_{B_1}w_{1,j})\otimes (w^*_{2,j}e_{B_2}w_{2,j})\right\|_{2,\Tr} < \vp.
\end{equation}

If we multiply on the right in \eqref{gpeq4.36} by $v^*e_B = v^*e_B vv^*e_B$, then the result is
\begin{equation}\label{gpeq4.37}
 \left\|v^*e_B - \left(\sum^k_{j=1} (w^*_{1,j}e_{B_1}w_{1,j}) \otimes (w^*_{2,j} e_{B_2}w_{2,j})\right) (v^*e_B)\right\|_{2,\Tr} < \vp,
\end{equation}
using the fact that $P_vv^*v = v^*e_Bv$. A typical element of the sum in \eqref{gpeq4.37} is
\[
 (w^*_{1j} \otimes w^*_{2,j}) (e_{B_1}\otimes e_{B_2}) (w_{1,j}\otimes w_{2,j})v^*(e_{B_1} \otimes e_{B_2})
\]
which has the form $(w^*_{1,j}\otimes w^*_{2,j})x_j^*e_B$, where $x_j^*=\mathbb E_B((w_{1,j}\otimes w_{2,j})v^*)\in B$ has $\|x_j^*\|\leq\|(w_{1,j}\otimes w_{2,j})v^*\|\leq 1$. Thus \eqref{gpeq4.37} becomes
\begin{equation}\label{gpeq4.38}
 \left\|(v^*-\sum_{j=1}^k(w_{1,j}\otimes w_{2,j})^*x_j^*)e_B\right\|_{2,\Tr} <\vp.
\end{equation}
For each $y\in M$,
\begin{equation}\label{gpeq4.39}
 \|ye_B\|^2_{2,Tr} = \Tr(e_By^*ye_B) =\tau(y^*y) = \|y\|^2_{2,\tau},
\end{equation}
and so \eqref{gpeq4.38} implies that 
$$
\left\|v-\sum_{j=1}^kx_j(w_{1,j}\otimes w_{2,j})\right\|_{2,\tau}<\vp,
$$ 
as required.
\end{proof}

\begin{cor}
 Let $B_i\subseteq M_i$, $i=1,2$, be inclusions of finite von Neumann algebras satisfying $B'_i \cap M_i \subseteq B_i$. Then
\begin{equation}\label{gpeq4.35}
 W^*(\oneGN (B_1\vnotimes B_2)) = W^*(\oneGN (B_1)) \vnotimes W^*(\oneGN (B_2)).
\end{equation}
\end{cor}

There are two extreme cases where the hypothesis $B'\cap M \subseteq B$ is satisfied. The first is when $B$ is an irreducible subfactor where the result of Theorem \ref{gpthm4.8} can be deduced from the stronger results of \cite{Saw.NormalisersSubfactors}. The second is when $B$ is a masa in $M$ where Theorem \ref{gpthm4.8} is already known, \cite{Chifan.Normalisers}. The following example explains the preference given to  intertwiners over unitary normalizers in intermediate cases, even in a simple setting.

\begin{exm}\label{gpexm4.9}
Let $M$ be a $\text{II}_1$ factor, let $p\in M$ be projection whose trace lies in $(0,1/2)$, and let $B = pMp + (1-p)M(1-p)$. This subalgebra has no nontrivial unitary normalizers, essentially because $\tau(p) \ne \tau(1-p)$. However, the tensor product $B\vnotimes B\subseteq M\vnotimes M$ does have such normalizers because the compressions by $p\otimes (1-p)$ and $(1-p)\otimes p$, which have equal traces, are conjugate by a unitary normalizer $u$ which is certainly outside $B\vnotimes B$. According to Theorem \ref{gpthm4.8}, $u$ can be obtained as the limit of finite sums of elementary tensors from $W^*(\oneGN (B))\vnotimes W^*(\oneGN (B))$.$\hfill\square$
\end{exm}

\section{Groupoid normalizers of tensor products}\label{gpsec5}

\indent

In this section we return to the groupoid normalizers $\twoGN_M(B)$, namely those $v\in M$ such that $v,v^*\in\oneGN_M(B)$. Our objective in this section is to establish a corresponding version of Theorem \ref{gpthm4.8} for $\twoGN(B)$, and consequently we will assume throughout that any inclusion $B\subseteq M$ satisfies the relative commutant condition $B'\cap M\subseteq B$.  

We will need to draw a sharp distinction between those intertwiners $v$ that are groupoid normalizers and those that are not, and so we introduce the following definition. 
\begin{defn}
Say that $v\in\oneGN (B)$ is {\emph{strictly one-sided}} if the only projection $p\in Z(Bv^*v) = Z(B)v^*v$ for which $vp\in \twoGN(B)$ is $p=0$.  When $B$ is an irreducible subfactor of $M$ then any unitary $u\in M$ satisfying $uBu^* \subsetneqq B$ is a strictly one-sided intertwiner (see \cite[Example 5.4]{Saw.NormalisersSubfactors} for examples of such unitaries). 
\end{defn}
Given $v\in\twoGN(B)$, recall from Section \ref{gpsec3} that there is a projection $P_v\in Z(B'\cap \langle M,e_B\rangle )$ such that $P_vv^*v = v^*e_Bv$, and $P_v$ has the form
\begin{equation}\label{gpeq5.1}
 P_v = \sum_{n\ge 0} v^*_n e_Bv_n,
\end{equation}
where there exist partial isometries $w_n\in B$ so that $v_n =vw^*_n \in \oneGN (B)$, and $w^*_mw_n = 0$ for $m\ne n$. Letting $p_n$ denote the projection $w_nw^*_n\in B$, it also holds that $v^*_nv_n = p_n$. We will employ this notation below.

\begin{lem}\label{gplem5.2}
Let $v\in \twoGN(B)$, and let $u\in \oneGN (B)$ be strictly one-sided. Then $P_vu^*e_Bu = 0$.
\end{lem}

\begin{proof}
By Lemmas \ref{gplem3.2} (i) and \ref{gplem4.3}, $u^*e_Bu\in Z(B'\cap \langle M,e_B\rangle)u^*u$ and $P_v \in Z(B'\cap \langle M,e_B\rangle)$, showing that $P_v$ and $u^*e_Bu$ are commuting projections. Let $Q$ denote the projection $P_vu^*e_Bu$ in $Z(B'\cap\langle M,e_B\rangle)u^*u$ which lies below both $P_v$ and $u^*e_Bu$. From Lemmas \ref{gplem3.6} and \ref{gplem4.5} we may find projections $z\in Z(B)$ and $f\in M\cap \{P_v\}'$ such that 
\begin{equation}\label{gpeq5.8}
 Q = zu^*e_Bu = fP_v.
\end{equation}
From \eqref{gpeq5.1}, write $P_v$ as the strongly convergent sum
\begin{equation}\label{gpeq5.9}
 P_v = \sum_{n\ge 0} w_nv^*e_Bvw^*_n,
\end{equation}
so that \eqref{gpeq5.8} becomes
\begin{equation}\label{gpeq5.10}
 \sum_{n\ge 0} fw_nv^*e_Bvw^*_n = zu^*e_Bu.
\end{equation}
If we multiply \eqref{gpeq5.10} on the right by $p_j=w_jw_j^*$ and on the left by $e_Bu$, noting that $uzu^*\in B$, then the result is 
\begin{align}
e_Bb_jvw^*_j = e_Buzu^*up_j = e_Buu^*uzp_j
\label{gpeq5.11}
= e_Buzp_j
\end{align}
for each $j\ge 0$, where $b_j = {\bb E}_B(ufw_jv^*)\in B$. Thus
\begin{equation}\label{gpeq5.12}
b_jvw^*_j = uzp_j,\qquad j\ge 0.
\end{equation}
If we sum \eqref{gpeq5.12} over $j\ge 0$, then the right-hand side will converge strongly, implying strong convergence of $\sum\limits_{j\ge 0} b_jvw^*_j$. If we return to \eqref{gpeq5.10} and multiply on the left by $e_Bu$, then we obtain
\begin{equation}\label{gpeq5.13}
uz = \sum_{n\ge 0} b_nvw^*_n
\end{equation}
with strong convergence of this sum. It follows that, for each $b\in B$,
\begin{equation}\label{gpeq5.14}
 zu^*buz = \lim_{k\to\infty} \sum_{m,n\le k} w_nv^*b^*_nbb_mvw^*_m
\end{equation}
strongly, and thus $zu^*buz\in B$ since $v^*b^*_nbb_mv\in B$. Thus the projection $p = zu^*u\in Z(Bu^*u)$ satisfies $pu^*Bup \subseteq B$. Since $u$ is strictly one-sided, we conclude that $zu^*u=0$, showing that
\begin{equation}\label{gpeq5.15}
 Q = zu^*e_Bu = 0
\end{equation}
from \eqref{gpeq5.8}. This proves the result.
\end{proof}  

We can use the preceding lemma to show that a strictly one-sided intertwiner $v\in\oneGN(B)$ has the property that the only projection $p\in Bv^*v$ for which $vp\in\twoGN(B)$ is $p=0$. Indeed, take such a projection $p$ for which $w=vp\in\twoGN(B)$.  Then $P_vw^*w=w^*e_Bw$ so that $P_v\geq P_w$ by Remark \ref{gprem3.7}. Lemma \ref{gplem5.2} then gives $P_wv^*e_Bv=0$. Thus $P_ww^*e_Bw=w^*e_Bw=0$ and so $w=0$.

\begin{lem}\label{gplem5.3}
 Let $v\in \oneGN (B)$. Then there exist orthogonal families of orthogonal projections $e_n,f_n\in v^*vBv^*v$ such that $\sum(e_n+f_n) = v^*v$, each $ve_n$ is strictly one-sided, and each $vf_n$ lies in $\twoGN(B)$.
\end{lem}

\begin{proof}
Let $\{e_n\}$ be a maximal orthogonal family of projections in $Z(B)v^*v$ such that $ve_n$ is strictly one-sided, and set $e=\sum e_n$. Then choose a maximal family of orthogonal projections $\{f_n\}\in Z(B)(v^*v-e)$ such that $vf_n\in \twoGN(B)$. If $\sum e_n + \sum f_n = v^*v$ then the result is proved, so consider the projection $g = v^*v - \sum e_n-\sum f_n \in Z(B)v^*v$ and suppose that $g\ne 0$. Then $vg$ cannot be strictly one-sided otherwise the maximality of $\{e_n\}$ would be contradicted. Thus there exists $z\in Z(B)$ such that $vgz$ is a nonzero element of $\twoGN(B)$. But this contradicts maximality of $\{f_n\}$, proving the result.
\end{proof}

We now return to considering two containments $B_i\subseteq M_i$ satisfying $B'_i\cap M_i\subseteq B_i$, and the tensor product containment $B = B_1\vnotimes B_2 \subseteq M_1\vnotimes M_2 = M$. The next lemma is the key step required to obtain a version of Theorem \ref{gpthm4.8} for groupoid normalizers. We need a result from the perturbation theory of finite von Neumann algebras. For any containment $A\subseteq N$, where $N$ has a specified trace $\tau$, recall that $N\subset_{\delta,\tau} A$ means that 
\[
 \sup \{\|x-{\bb E}_A(x)\|_2\colon \ x\in N, \ \|x\|\le 1\} \le \delta
\]
where $\|\cdot\|_2$ is defined using the given trace $\tau$. If $\tau$ is scaled by a constant $\lambda$, then $\subset_{\delta,\lambda\tau}$ is the same as $\subset_{\delta/\sqrt\lambda,\tau}$. Then \cite[Theorem 3.5]{Sinclair.PertSubalg} (see also \cite[Theorem 10.3.5]{Sinclair.MasaBook}), stated for normalized traces, has the following general interpretation:\ if $A\subseteq N$ and $N\subset_{\delta,\tau}A$ for some $\delta < (\tau(1)/23)^{1/2}$, then there exists a nonzero projection $p\in Z(A'\cap N)$ such that $Ap = pNp$.

\begin{lem}\label{gplem5.4}
Let $v\in \oneGN (B_1)$ and $w\in \oneGN (B_2)$. If $v$ (or $w$) is strictly one-sided then $v\otimes w\in \oneGN (B)$ is strictly one-sided.
\end{lem}

\begin{proof}
Without loss of generality, suppose that $v$ is strictly one-sided. Fix a nonzero projection $p\in Z(B)(v^*v\otimes w^*w)$, and let $\tau_i$ be the faithful normalized normal trace on $M_i\supset B_i$. Given $\vp>0$, we may choose projections $p_i\in Z(B_1)v^*v$ and $q_i\in Z(B_2)w^*w$, $1\le i\le k$, with $\|p - \sum\limits^k_{i=1} p_i\otimes q_i\|_{2,\tau} < \vp$ and the $p_i$'s orthogonal since these projections lie in abelian von Neumann algebras. Here, $\|\cdot\|_{2,\tau}$ is with respect to the normalized trace $\tau = \tau_1\otimes \tau_2$ on $M$. Now
\begin{equation}\label{gpeq5.16}
 p_iB_1p_i = p_iv^*vB_1v^*vp_i \subseteq p_iv^*B_1vp_i
\end{equation}
since $vB_1v^*\subseteq B_1$. If it were true that $p_iv^*B_1vp_i \subset_{\delta,\tau_1}p_iB_1p_i$ for some $\delta < (\tau(p_i)/23)^{1/2}$, then it would follow from \cite[Theorem 3.5]{Sinclair.PertSubalg} that there exists a nonzero projection $p'_i\in (p_iB_1p_i)'\cap p_iv^*B_1vp_i \subseteq Z(B_1)p_i$ such that $p'_ip_iB_1p_i = p'_ip_iv^*B_1vp_ip'_i$, contradicting the hypothesis that $v$ is strictly one-sided. Thus there exists $b_i\in B_1$ satisfying
\begin{equation}\label{gpeq5.17}
 \|p_iv^*b_ivp_i\|\le 1,\quad d_{2,\tau_1}(p_iv^*b_ivp_i, p_iB_1p_i) \ge (\tau_1(p_i)/23)^{1/2},
\end{equation}
where $d_{2,\tau_1}(x,A) = \inf\{\|x-a\|_{2,\tau_1}\colon \ a\in A\}$ for any von~Neumann algebra $A$.

Since each $p_i$ lies under $v^*v$, the projections $vp_iv^*$ lie in $B_1$ and are orthogonal. We may then define an element $b\in B$ by
\begin{equation}\label{gpeq5.18}
 b = \sum_i vp_iv^*b_ivp_iv^* \otimes wq_iw^*,
\end{equation}
and the orthogonality gives $\|b\|\le 1$. Moreover,
\begin{equation}\label{gpeq5.19}
(v\otimes w)^* b(v\otimes w) = \sum_i p_iv^*b_ivp_i\otimes q_i,
\end{equation}
and
\begin{align}
d_{2,\tau}\left((v\otimes w)^* b(v\otimes w), B\right)^2 &\ge \left(\sum_i \tau_1(p_i)\tau_2(q_i)\right)/23\notag\\
\label{gpeq5.20}
&> (\tau(p)-\vp)/23.
\end{align}
The right-hand side of \eqref{gpeq5.19} is unchanged by pre- and post-multiplication by the projection $\sum\limits_i p_i\otimes q_i$, and this is close to $p$. This leads to the estimate
\begin{equation}\label{gpeq5.21}
 \|p(v\otimes w)^* b(v\otimes w)p - (v\otimes w)^* b(v\otimes w)\|_{2,\tau} \le 2\left\|p - \sum_i p_i\otimes q_i\right\|_{2,\tau} < 2\vp.
\end{equation}
From \eqref{gpeq5.20} and \eqref{gpeq5.21},
\begin{equation}\label{gpeq5.22}
 d_{2,\tau}(p(v\otimes w)^* b(v\otimes w)p, B) > ((\tau(p)-\vp)/23)^{1/2} -2\vp.
\end{equation}
A sufficiently small choice of $\vp$ then shows that $p(v\otimes w)^* b(v\otimes w)p\notin B$, and thus $v\otimes w$ is strictly one-sided.
\end{proof}

We can now give the two-sided counterpart of Theorem \ref{gpthm4.8}.

\begin{thm}\label{gpthm5.5}
 Let $B_i\subseteq M_i$, $i=1,2$, be inclusions of finite von Neumann algebras satisfying $B'_i \cap M_i \subseteq B_i$. Given $v\in\twoGN_{M_1\vnotimes M_2}(B_1\vnotimes B_2)$ and $\vp>0$, there exist $x_1,\dots,x_k\in B_1\vnotimes B_2$, $w_{1,1},\dots,w_{1,k}\in\twoGN_{M_1}(B_1)$ and $w_{2,1},\dots,w_{2,k}\in\twoGN_{M_2}(B_2)$ such that:
 \begin{enumerate}
 \item $\|x_j\|\leq 1$ for each $j$;
\item\begin{equation}
\left\|v-\sum_{j=1}^kx_j(w_{1,j}\otimes w_{2,j})\right\|_{2,\tau}<\vp.
\end{equation}
\end{enumerate}
\end{thm}

\begin{proof}
Consider $v\in \twoGN(B)$. Following the proof of Theorem \ref{gpthm4.8}, given $\vp>0$, there exist $v_{i,j}\in \oneGN (B_i)$, $1\le j\le k$, so that
\begin{equation}\label{gpeq5.23}
\left\|P_v - \sum^k_{j=1} (v_{1,j}\otimes v_{2,j})^* e_B(v_{1,j}\otimes v_{2,j})\right\|_{2,\Tr} < \vp
\end{equation}
as in \eqref{gpeq4.36}. Using Lemma \ref{gplem5.3}, we may replace this sum with one of the form
\[
 \sum(w_{1,j} \otimes w_{2,j})^* e_B(w_{1,j}\otimes w_{2,j}) + \sum(x_{1,j}\otimes x_{2,j})^* e_B(x_{1,j} \otimes x_{2,j})
\]
where the $w_{i,j}$'s are two-sided and, for each $j$, at least one of $x_{1,j},x_{2,j}$ is strictly one-sided.  By Lemma \ref{gplem5.4}, each $x_{1,j}\otimes x_{2,j}$ is strictly one-sided, so $(x_{1,j} \otimes x_{2,j})^* e_B(x_{1,j}\otimes x_{2,j})P_v = 0$ by Lemma \ref{gplem5.2}. If we multiply on the right by $P_v$, then
\begin{equation}\label{gpeq5.24}
\left\|P_v - \sum_j(w_{1,j}\otimes w_{2,j})^* e_B(w_{1,j}\otimes w_{2,j})P_v\right\|_{2,\tau} < \vp.
\end{equation}
Simple approximation allows us to obtain the same estimate for some finite subcollection of the $w_{1,j}$ and $w_{2,j}$, say $w_{1,1},\dots,w_{1,k}$ and $w_{2,1},\dots,w_{2,k}$.  We now continue to follow the proof of Theorem \ref{gpthm4.8} from (\ref{gpeq4.36}) to obtain the required $x_j$.
\end{proof}

Just as in Section \ref{gpsec4}, the next corollary follows immediately.
\begin{cor}\label{TwoCor}
Let $B_i\subseteq M_i$, $i=1,2$, be inclusions of finite von Neumann algebras satisfying $B'_i \cap M_i \subseteq B_i$. Then
$$
\twoGN_{M_1}(B_1)''\vnotimes \twoGN_{M_2}(B_2)''=\twoGN_{M_1\vnotimes M_2}(B_1\vnotimes B_2)''.
$$
\end{cor}

\section*{Author Addresses}

\begin{tabular*}{\textwidth}{l@{\hspace*{2cm}}l}
Junsheng Fang&Roger Smith\\
Department of Mathematics&Department of Mathematics\\
University of New Hampshire&Texas A\&M University\\
Durham, NH 03824&College Station, TX 77843\\
USA&USA\\
\texttt{jfang@cisunix.unh.edu}&\texttt{rsmith@math.tamu.edu}
{}\\\\
Stuart White&Alan Wiggins\\
Department of Mathematics&Department of Mathematics\\
University of Glasgow&Vanderbilt University\\
Glasgow G12 8QW&Nashville, TN 37240\\
UK&USA\\
\texttt{s.white@maths.gla.ac.uk}&\texttt{alan.d.wiggins@vanderbilt.edu}
\end{tabular*}

\end{document}